\documentclass[a4paper]{amsart}

\usepackage{pdfpages}
\usepackage{amssymb,amsmath}

\theoremstyle{plain}
\newtheorem{definition}{Definition}[section]
\newtheorem{remark}[definition]{Remark}
\newtheorem{theorem}[definition]{Theorem}
\newtheorem{lemma}[definition]{Lemma}

\newtheorem{proposition}[definition]{Proposition}

\DeclareMathOperator{\Pol}{Pol}

\DeclareMathOperator{\ext}{ext}

\DeclareMathOperator{\Card}{Card}

\DeclareMathOperator{\img}{Im}

\begin{document}
%\includepdf[pages=\startpage-\endpage,
%scale=\scalefactor]{\infilename}

\title{Meet-reducible submaximal clones determined by two central relations.}

\author[YANNICK L.~T. JEUFACK, Luc E.~F. Diekouam and ETIENNE R.~A. Temgoua]{Y.~L.~T. JEUFACK, L. Diekouam and E.~R.~A. Temgoua}
\email{retemgoua@yahoo.fr; tenkius2000@yahoo.fr;
lucdiekouam@yahoo.fr}
\address{Department of Mathematics\\Ecole Normale Sup\'erieure\\
University of Yaound\'e-1\\P.O.Box 47 Yaound\'e\\CAMEROON\\
Department
of mathematics\\ Ecole Normale Sup\'erieure\\ University of Maroua\\
P.O. Box 55 Maroua\\ Cameroon}


\begin{abstract}  Let $\rho$ and $\sigma$ be two central relations on a finite set $A$.
It is known from Rosenberg's classification theorem (1965) that the
clones $\Pol\rho$ and $\Pol\sigma$ which consists of all operations
on $A$ that preserve $\rho$  respectively $\sigma$  are among the
maximal clones on $A$. In this paper, we find all central relations
$\sigma$ such that the clone $\Pol\{\rho, \sigma\}$ is a maximal
subclone of $\Pol\rho$  where $\rho$ is a fixed central relation.
\end{abstract}
\maketitle
\section{Introduction}
 In 1941, E.L. Post presented
the complete description of the countably many clones on 2 elements.
It turned out that, all such clones are finitely generated and the
lattice of these clones is countable. The structure of the lattice
of clones on finitely many (but more than 2) elements is more
complex and is of the cardinality $2^{\aleph_{0}}$. For $k\geq 3$,
not much is known about the structure of the lattice of clones in
spite of the efforts made by many researchers in this area.
Therefore, every new piece of information  is considered valuable.
Indeed, it would be very interesting to know the clone lattice on
the next level (below the maximal clones) and even a partial
description will shed more light onto its structure. The complete
description of all submaximal clones is known only for the 2-element
case and the 3-element case (see \cite{Lau 1,Lau 3,Post1}), however
the result in (\cite{Lau 1}) and many result in the literature on
clones including those discussed in (\cite{Lau 1,Lau 3,Rosenb 1,temg
1,Temg 2}), require intensive knowledge of submaximal clone (below
certain maximal clones) on arbitrary finite sets. Clone theory is
considered to be very important because of its use to understand
universal algebras.

In \cite[Chapter~17]{Lau 3}, D. Lau presented all submaximal clones
of the clone $\Pol\rho$ where $\rho$ is a unary central relation on
an arbitrary finite set.
 In this paper, we characterize the five types of central relations  $\sigma$
 such that the clone of the form
$\Pol\{\rho, \sigma\}$ is  covered by $\Pol\rho$, where $\rho$ is a
$h$-ary $(h\geq 2)$ central relation on a given finite set.
Moreover, we give a result which will help anyone to decide whether
$\Pol\{\rho, \sigma\}$ is a submaximal clone where $\rho$ and
$\sigma$ are two central relations.

This paper consists of four sections. After this Introduction, in
which we motivated this research and we announce the five types of
central relations to be characterized in the paper, the second
section provides the reader with necessary notions and notations. It
is followed by the section dedicated to the description of the five
types of central relations that are in the focus of the study. It is
also the place where the main result of the paper is stated, and
proved in one direction (the sufficient condition). The final
section contains the proof that the given conditions are also
necessary.

\section{Preliminaries}
In this section, we provide the reader with some basic notions and notations; for more details
 the reader can see (\cite{Lau 3,A. Szendrei,temg 1,Temg 2}).

Let $A$ be a fixed finite set with $k$ elements, $n$ and $h$ be
integers
 such that $1\leq n,h$. An $n$-ary operation on $A$ is a function  $f:A^{n}\rightarrow A$. We will use the notation
$\mathcal{O}_{A}^{(n)}$ for the set of all $n$-ary operations on
$A$, and  $\mathcal{O}_{A}$ for the set
$\underset{n\geq1}{\cup}\mathcal{O}_{A}^{(n)}$ of all finitary
operations on A. For  $1\leq i\leq n$, the $i$-th projection is the
operation $\pi_{i}^{(n)}: A^{n}\rightarrow A,
(a_{1},\ldots,a_{n})\mapsto a_{i}$. For arbitrary positive integers
$m$ and $n$, there is a one-to-one correspondence between the
functions $f: A^{n}\rightarrow A^{m}$ and the $m$-tuples
$\mbox{\boldmath{$f$}}=(f_{1},\ldots,f_{m})$ of functions $f_{i}:
A^{n}\rightarrow A$ (for $i=1,\ldots,m$) via $f\mapsto
\mbox{\boldmath{$f$}}=(f_{1},\ldots,f_{m})$ with
$f_{i}=\pi_{i}^{(m)}\circ f$ for all $i=1,\ldots,m$. In particular,
$\pi^{(n)}=(\pi_{1}^{(n)},\ldots,\pi_{n}^{(n)})$ corresponds to the
identity function $f: A^{n}\rightarrow A^{n}$. From now on, we will
identify each function $f: A^{n}\rightarrow A^{m}$ with the
corresponding $m$-tuples
$\mbox{\boldmath{$f$}}=(f_{1},\ldots,f_{m})\in
(\mathcal{O}_{A}^{(n)})^{m}$ of $n$-ary operations. Using this
convention, the composition of two functions
$\mbox{\boldmath{$f$}}=(f_{1},\ldots,f_{m}): A^{n}\rightarrow A^{m}$
and $\mbox{\boldmath{$g$}}=(g_{1},\ldots,g_{p}): A^{m}\rightarrow
A^{p}$ can be described as follows:
$\mbox{\boldmath{$g$}}\circ\mbox{\boldmath{$f$}}=(g_{1}\circ\mbox{\boldmath{$f$}},\ldots,g_{p}\circ\mbox{\boldmath{$f$}})=
(g_{1}(f_{1},\ldots,f_{m}),\ldots,g_{p}(f_{1},\ldots,f_{m}))$ where
$g_{i}(f_{1},\ldots,f_{m})(\mbox{\boldmath{$a$}})=g_{i}(f_{1}(\mbox{\boldmath{$a$}}),\ldots,f_{m}(\mbox{\boldmath{$a$}}))$
for all $\mbox{\boldmath{$a$}}\in A^{n}$ and $1\leq i\leq p$.

A clone on $A$ is a subset $\mathcal{C}$ of $\mathcal{O}_{A}$ that
contains the projections and is closed  under composition; that is
$\pi_{i}^{(n)}\in \mathcal{C}$ for all $n\geq 1$ and $1\leq i\leq
n$, and $g\circ \mbox{\boldmath{$f$}}\in \mathcal{C}^{(n)}$ whenever
$g\in \mathcal{C}^{(m)}$ and $\mbox{\boldmath{$f$}}\in
(\mathcal{C}^{(n)})^{m}$ (for $m,n\geq 1$). The clones on $A$ form a
complete lattice $\mathcal{L}_{A}$ under inclusion. Therefore, for
each set $F\subseteq \mathcal{O}_{A}$ of operations, there exists a
smallest clone that contains $F$, which will be denoted by $\langle
F\rangle$ and will be called clone generated by $F$. Clones can also
be described via invariant relations.
 An $h$-ary relation on $A$ is a subset of $A^{h}$. For an $n$-ary operation $f\in
\mathcal{O}_{A}^{(n)}$ and an $h$-ary relation $\rho$  on $A$, we
say that $f$ preserves $\rho$ (or $\rho$ is invariant under $f$, or
$f$ is a polymorphism of $\rho$) if whenever $f$ is applied
coordinatewise to $h$-tuples from $\rho$, the resulting $h$-tuple
belongs to $\rho$ i.e., for all $(a_{1,i},\ldots,a_{h,i})\in\rho,
i=1,\ldots,n$,
$(f(a_{1,1},\ldots,a_{1,n}),f(a_{2,1},\ldots,a_{2,n})\ldots,f(a_{h,1},\ldots,a_{h,n}))\in\rho$.
 For any family $\mathcal{R}$ of (finitary) relations on $A$, the set
$\Pol \mathcal{R}$ of all operations $f\in \mathcal{O}_{A}$ that
preserve each relation in $\mathcal{R}$  is easily seen to be a
clone on $A$. Moreover,
 if $A$ is finite, then it is a well-known fact that every clone
on $A$ is of the form $\Pol \mathcal{R}$ for some family
$\mathcal{R}$ of relations on $A$. If $\mathcal{R}=\{\rho\}$, we
write $\Pol\rho$ for $\Pol\{\rho\}$. Let $\rho\subseteq A^{h}$; for
an integer $m>1$ and
$\mbox{\boldmath{$a$}}_{i}=(a_{1,i},\ldots,a_{m,i})\in A^{m},1\leq
i\leq h$, we will write
$(\mbox{\boldmath{$a$}}_{1},\ldots,\mbox{\boldmath{$a$}}_{h})\in\rho$
if for all $j\in\{1,\ldots,m\}, (a_{j,1},\ldots,a_{j,h})\in\rho$.

Since $A$ is finite, it is well known that every clone on $A$ other
than $\mathcal{O}_{A}$ is contained in a maximal clone. We say that
an $h$-ary relation  $\rho$ on $A$  is totally reflexive (reflexive
for $h=2$) if $\rho$ contains the $h$-ary relation $\iota_{A}^{h}$
defined by
\[
\iota_{A}^{h}=\{(a_{1},\ldots,a_{h})\in A^{h}| \exists
i,j\in\{1,\ldots,h\}: i\neq j~\mbox{and}~ a_{i}= a_{j}\},
\]
and is totally symmetric (symmetric if $h=2$) if $\rho$ is invariant
under any permutation of its coordinates. If $\rho$ is  totally
reflexive and totally symmetric, we define the center of $\rho$,
denoted by $C_{\rho}$, as follows:
\[ C_{\rho}=\{a\in
A:(a,a_{2},\ldots,a_{h})\in\rho~\mbox{for all}~a_{2},\ldots,a_{h}\in
A\}.\]

We say that $\rho$ is a  central relation  if $\rho$  is totally
reflexive, totally symmetric and has a nonvoid center which is a
proper subset of $A$. It is known of course that whenever $\rho$ and
$\sigma$ are distinct nontrivial central relations, $\Pol\rho$ and
$\Pol\sigma$ are distinct maximal clones. Let $\rho$ be a binary
relation on $A$, $\rho$ is an equivalence relation if $\rho$ is
symmetric, reflexive and transitive; $\rho$ is non-trivial
 if $\rho\neq A^{2}$ and $\rho\neq \{(a,a): a\in A\}$.
For instance, it is nice for us to give the following remark useful
to justify some inclusions between clones.

\begin{remark}
Let $R$ be a set of relations on a finite set $A$. If $f\in \Pol R$,
then $f\in\Pol [R]$ where $[R]$ is the relational clone generated by
$R$.
\end{remark}

For two clones $\mathcal{C}$ and $\mathcal{D}$ on $A$, we say that
$\mathcal{C}$ is maximal in $\mathcal{D}$ if $\mathcal{D}$ covers
$\mathcal{C}$ in $\mathcal{L}_{A}$, we also say that $\mathcal{C}$
is submaximal if $\mathcal{C}$ is maximal in a clone $\mathcal{D}$
and $\mathcal{D}$ is a maximal clone on $A$. For a maximal clone
$\mathcal{D}$, there are two types of clones $\mathcal{C}$ being
maximal in $\mathcal{D}$: $\mathcal{C}$ is meet-reducible if
$\mathcal{C}=\mathcal{D}\cap \mathcal{F}$ for a maximal clone
$\mathcal{F}$ distinct from $\mathcal{D}$ (but not necessarily
unique) and $\mathcal{C}$ is meet-irreducible if it is not
meet-reducible.

From now on we assume that we are working on the set
$E_{k}=\{0,1,\ldots,k-1\}$ where $k>1$. We will denote by
$\underline{h}$ the set $\{1,\ldots,h\}$ and by  $S_{h}$ the
 set of all permutations on $\underline{h}$ for all integers $h>1$.
 For any integer $ 2\leq h\leq k$, we denote by $\iota_{k}^{h}$ the set $\iota_{E_{k}}^{h}$.
It is well known (see \cite{Lau 3}) that the Slupecki clone
$\Pol\iota_{k}^{k}$ is a maximal clone.

\section{The five types of $\sigma$ such that $\Pol\{\rho, \sigma\}$ is maximal in $\Pol\rho$.}

In this section, we give the definition of those types of central
relations $\sigma$ such that $\Pol\{\rho, \sigma\}$ is maximal in
$\Pol\rho$. We recall some classical constructions. If $\alpha$ and
$\beta$ are two $h$-ary relations on $E_{k}$, the intersection of
the relations $\alpha$ and $\beta$, denoted by $\alpha\cap\beta$, is
the set: \[\alpha\cap\beta=\{(a_{1},\ldots,a_{h})\in E_{k}^{h}:
(a_{1},\ldots,a_{h})\in\alpha\wedge
(a_{1},\ldots,a_{h})\in\beta\}.\] If $\alpha$ is an $h$-ary relation
($h\geq 2$), we denote by $\alpha_{1}$, the relation
\[\{(x_{1},\ldots,x_{h})\in E_{k}^{h}:\exists u\in
E_{k}, \forall 1\leq i< h,
(x_{1},\ldots,x_{i-1},u,x_{i+1}\ldots,x_{h})\in\alpha)\}.\] Since
$\alpha\cap\beta$ and $\alpha_{1}$ belong to $[\{\alpha,\beta\}]$,
we have $\Pol\{\alpha,\beta\}\subseteq\Pol(\alpha\cap\beta)$ and
$\Pol\{\alpha, \beta\}\subseteq\Pol\alpha_{1}$.
 \begin{definition}
Let $h\geq 1$, $\rho$ and $\sigma$ be two $h$-ary relations on
$E_{k}$ such that $\rho\neq\sigma$.
\begin{itemize}
\item[(1)] We say that $\rho$ and $\sigma$ are comparable if $\rho\subseteq\sigma$ or
$\sigma\subseteq\rho$.
\item[(2)] A subset $B$ of  $E_{k}$ is called a $\rho$-chain if $B^{h}$ is a subset of $\rho$.
\item[(3)] A $\rho$-chain $B$ is called a maximal $\rho$-chain if $B$ is not a proper
subset of another  $\rho$-chain $D$.
\end{itemize}
\end{definition}
It is easy to check that for all $\rho$-chains $B$, there is a
maximal $\rho$-chain $D$ such that $B\subseteq D$. In the following
lines, $\rho$ is an $h$-ary central relation and $\sigma$ is an
$s$-ary central relation on $E_{k}$.
 If $2\leq h<s$, we consider the $s$-ary relations $\lambda$ and $\gamma'$ defined on $E_{k}$ by
$\lambda=\{(a_{1},\ldots,a_{s})\in E_{k}^{s}:
(a_{1},\ldots,a_{h})\in\rho\}~\text{ and }
~\gamma'=\lambda\cap\sigma$. If $\sigma$ is a unary central relation
and $h=2$, we consider the binary relation $\gamma$ defined on
$E_{k}$ by: $\gamma=\{(a,b)\in E_{k}^{2}: \exists u\in\sigma,
(a,u)\in\rho\wedge (b,u)\in\rho\}.$

Here we state the main result of this paper.
\begin{theorem}\label{theo 1}
Let $k\geq 3$, $\rho$ be a central relation on $E_{k}$ with arity
$h\geq 2$ and $\sigma$  an $s$-ary central relation on $E_{k}$ such
that $\sigma\neq\rho$ . $\Pol\{\rho, \sigma\}$ is a maximal clone
below $\Pol\rho$ if and only if $\sigma$ fulfils one of the
following five conditions:
\begin{itemize}
\item[(I)] $\sigma$ is unary and $C_{\rho}\cap\sigma\neq
\emptyset$;
\item[(II)] $\sigma$ is unary, $\rho=\{(a,b)\in E_{k}^{2}: \exists u\in \sigma, (a,u)\in\rho\wedge (b,u)\in\rho\}$
and for all maximal $\rho$-chains $B$, $B\cap\sigma\neq \emptyset$;
\item[(III)] $s=h$, $\rho$ and $\sigma$ are comparable (i.e.
$\rho\subsetneq\sigma$ or $\sigma\subsetneq\rho)$;
\item[(IV)] $2\leq s<h$ and $C_{\rho}\cap C_{\sigma}\neq\emptyset$;
\item[(V)] $2\leq h< s$ and $\lambda=\{(a_{1},\ldots,a_{s})\in E_{k}^{s}:(a_{1},\ldots,a_{h})\in\rho\}\subsetneq\sigma$.
\end{itemize}
\end{theorem}

The proof of Theorem \ref{theo 1}  is organized as follows. The
sufficiency of conditions is shown in Proposition \ref{prop 1} and
the necessity is shown in Propositions \ref{prop 2}, \ref{prop 3},
\ref{prop 4} and \ref{prop 5}.
\begin{definition}
Let $l\in\{I,II,III,IV,V\}$. We say that $\sigma$ is of type $l$ if
$\sigma$ satisfies the condition  $(l)$ of Theorem \ref{theo 1}.
\end{definition}
Now we will prove the sufficiency of the conditions in Theorem \ref{theo 1}.

\begin{proposition}\label{prop 1}
Let $k\geq 3$, $\rho$ be an h-ary central relation ($2\leq h$) and
$\sigma$ an s-ary  central relation ($s\geq 1$). If $\sigma$ is of
type $l\in\{I,II,III,IV,V\}$, then $\Pol\{\rho, \sigma\}$ is maximal
in $\Pol\rho$.
\end{proposition}

Let $g\in\Pol\rho\setminus\Pol\sigma$ be an $n$-ary operation; then
there exist
$\mbox{\boldmath{$a$}}_{1}=(a_{1,1},\ldots,a_{s,1})$,$\ldots$,
$\mbox{\boldmath{$a$}}_{n}=(a_{1,n},\ldots,a_{s,n})\in\sigma$ such
that
\[g(\mbox{\boldmath{$a$}}_{1},\ldots,\mbox{\boldmath{$a$}}_{n})=(g(a_{1,1},\ldots,a_{1,n}),
\ldots,g(a_{s,1},\ldots,a_{s,n}))\not\in~\sigma.\]

\begin{lemma}\label{lem 1}
Let $m\geq 1$ be an integer, $k\geq 3$ and $\rho$ and $\sigma$ be
central relations of arity $h\geq 2$ and $s\geq 1$, respectively.
Moreover let $g\in\Pol\rho\setminus\Pol\sigma$ and
$\mbox{\boldmath{$a$}}_{1},\ldots,\mbox{\boldmath{$a$}}_{n}\in\sigma$
with
$g(\mbox{\boldmath{$a$}}_{1},\ldots,\mbox{\boldmath{$a$}}_{n})\notin\sigma$.
\begin{itemize}
\item[(i)] If $\sigma$ is of type I or II, then for all $\mbox{\boldmath{$c$}}=(c_{1},\ldots,c_{m})\in\sigma^{m}$
there exists an $m$-ary operation $f_{\mbox{\boldmath{$c$}}}\in
\langle (\Pol\{\rho, \sigma\})\cup\{g\}\rangle$ such that
$f_{\mbox{\boldmath{$c$}}}(\mbox{\boldmath{$c$}})\not\in\sigma$.
\item[(ii)] If $\sigma\subsetneq\rho$ or $\sigma$ is of type IV, then for all $\mbox{\boldmath{$c$}}_{1}=(c_{1,1},\ldots,c_{s,1}),\ldots,\mbox{\boldmath{$c$}}_{m}=(c_{1,m},
\ldots,c_{s,m})\in E_{k}^{s}$
 such that $\mbox{\boldmath{$b$}}_{1}=(c_{1,1},\ldots,c_{1,m}),\ldots,
    \mbox{\boldmath{$b$}}_{s}=(c_{s,1},\ldots,c_{s,m})$ are pairwise distinct elements of $E_{k}^{m}$,
 there exists an $m$-ary operation

 $f_{\mbox{\boldmath{$c$}}_{1},\ldots,\mbox{\boldmath{$c$}}_{m}}\in~\langle (\Pol\{\rho, \sigma\})\cup\{g\}\rangle$ with
$f_{\mbox{\boldmath{$c$}}_{1},\ldots,\mbox{\boldmath{$c$}}_{m}}(\mbox{\boldmath{$c$}}_{1},\ldots,\mbox{\boldmath{$c$}}_{m})\not\in\sigma$.
\item[(iii)] If $\rho\subsetneq \sigma$ or $\sigma$ is of type V, then for all $\mbox{\boldmath{$c$}}_{1}=(c_{1,1},\ldots,c_{s,1}),\ldots,
\mbox{\boldmath{$c$}}_{m}=(c_{1,m},\ldots,c_{s,m})$ such that
$(\mbox{\boldmath{$b$}}_{i_{1}},\ldots,\mbox{\boldmath{$b$}}_{i_{h}})\not\in\rho$
for every $1\leq i_{1}< i_{2}<\cdots<i_{h}\leq s $ (where
$\mbox{\boldmath{$b$}}_{j}=(c_{j,1},\ldots,c_{j,m})$ for $1\leq
j\leq s$), there exists an $m$-ary operation
 $f_{\mbox{\boldmath{$c$}}_{1},\ldots,\mbox{\boldmath{$c$}}_{m}}\in \langle(\Pol\{\rho,\sigma\})\cup\{g\}\rangle$ such that  $f_{\mbox{\boldmath{$c$}}_{1},\ldots,
 \mbox{\boldmath{$c$}}_{m}}(\mbox{\boldmath{$c$}}_{1},\ldots,\mbox{\boldmath{$c$}}_{m})\not\in~\sigma$.
\end{itemize}
\end{lemma}

\begin{proof}
i) Let $\mbox{\boldmath{$c$}}=(c_{1},\ldots,c_{m})\in\sigma^{m}$ and
consider the $m$-ary function $f_{\mbox{\boldmath{$c$}}}^{i},1\leq
i\leq n$ defined by
$f_{\mbox{\boldmath{$c$}}}^{i}(\mbox{\boldmath{$x$}})=\mbox{\boldmath{$a$}}_{i}$.
Since $\mbox{\boldmath{$a$}}_{i}\in\sigma$ and $\rho$ is totally
reflexive, we have $f_{\mbox{\boldmath{$c$}}}^{i}\in \Pol\{\rho,
\sigma\}$. Set
$f_{\mbox{\boldmath{$c$}}}=g(f_{\mbox{\boldmath{$c$}}}^{1},\ldots,f_{\mbox{\boldmath{$c$}}}^{n})$;
$f_{\mbox{\boldmath{$c$}}}\in\langle (\Pol\{\rho,
\sigma\})\cup\{g\}\rangle$ and
$f_{\mbox{\boldmath{$c$}}}(\mbox{\boldmath{$c$}})
=g(f_{\mbox{\boldmath{$c$}}}^{1}(\mbox{\boldmath{$c$}}),\ldots,f_{\mbox{\boldmath{$c$}}}^{n}(\mbox{\boldmath{$c$}}))=g(\mbox{\boldmath{$a$}}_{1},\ldots,
\mbox{\boldmath{$a$}}_{n})\not\in\sigma$.

ii) In this case we have $s\leq h$. Let
$\mbox{\boldmath{$c$}}_{1}=(c_{1,1},\ldots,c_{s,1})$,$\ldots$,$\mbox{\boldmath{$c$}}_{m}=(c_{1,m},\ldots,c_{s,m})$
in $ E_{k}^{s}$ such that
$\mbox{\boldmath{$b$}}_{1}=(c_{1,1},\ldots,c_{1,m}),\ldots,\mbox{\boldmath{$b$}}_{s}=(c_{s,1},\ldots,c_{s,m})$
are pairwise distinct elements of $E_{k}^{m}$. For $1\leq i\leq n$,
we consider the $m$-ary function
$f_{\mbox{\boldmath{$c$}}_{1},\ldots,\mbox{\boldmath{$c$}}_{m}}^{i}$
defined by
\[f_{\mbox{\boldmath{$c$}}_{1},\ldots,\mbox{\boldmath{$c$}}_{m}}^{i}(\mbox{\boldmath{$x$}})=
\begin{cases}
a_{j,i} & \text{if }\mbox{\boldmath{$x$}}=\mbox{\boldmath{$b$}}_{j} \text{ for some } 1\leq j\leq s, \\
a_{1,i} & \text{otherwise.}
\end{cases}\]
Since $(a_{1,i},\ldots,a_{s,i})\in\sigma$ and $\sigma\subsetneq\rho$
or $h> s$, we have
$f_{\mbox{\boldmath{$c$}}_{1},\ldots,\mbox{\boldmath{$c$}}_{m}}^{i}\in
\Pol\{\rho,\sigma\}$ for all $1\leq i\leq n$.
 Set $f_{\mbox{\boldmath{$c$}}_{1},\ldots,\mbox{\boldmath{$c$}}_{m}}=g(f^{1}_{\mbox{\boldmath{$c$}}_{1},\ldots,\mbox{\boldmath{$c$}}_{m}},\ldots,f^{n}_{\mbox{\boldmath{$c$}}_{1},\ldots,\mbox{\boldmath{$c$}}_{m}})$.
  We have $f_{\mbox{\boldmath{$c$}}_{1},\ldots,\mbox{\boldmath{$c$}}_{m}}\in \langle (\Pol\{\rho,\sigma\})\cup\{g\}\rangle$ and
\begin{eqnarray*}
  f_{\mbox{\boldmath{$c$}}_{1},\ldots,\mbox{\boldmath{$c$}}_{m}}(\mbox{\boldmath{$c$}}_{1},\ldots,\mbox{\boldmath{$c$}}_{m}) &=& g(f^{1}_{\mbox{\boldmath{$c$}}_{1},
  \ldots,\mbox{\boldmath{$c$}}_{m}}(\mbox{\boldmath{$c$}}_{1},
  \ldots,\mbox{\boldmath{$c$}}_{m}),\ldots,f^{n}_{\mbox{\boldmath{$c$}}_{1},\ldots,\mbox{\boldmath{$c$}}_{m}}(\mbox{\boldmath{$c$}}_{1},
  \ldots,\mbox{\boldmath{$c$}}_{m})) \\
  &=& g((f^{1}_{\mbox{\boldmath{$c$}}_{1},\ldots,\mbox{\boldmath{$c$}}_{m}}(\mbox{\boldmath{$b$}}_{1}),\ldots,f^{1}_{\mbox{\boldmath{$c$}}_{1},\ldots,
  \mbox{\boldmath{$c$}}_{m}}(\mbox{\boldmath{$b$}}_{s})),\ldots,(f^{n}_{\mbox{\boldmath{$c$}}_{1},\ldots,\mbox{\boldmath{$c$}}_{m}}(\mbox{\boldmath{$b$}}_{1}),\\
  & &\ldots,f^{n}_{\mbox{\boldmath{$c$}}_{1},\ldots,\mbox{\boldmath{$c$}}_{m}}(\mbox{\boldmath{$b$}}_{s}))) \\
   &=& g((a_{1,1},\ldots,a_{s,1}),\ldots,(a_{1,n},\ldots,a_{s,n})) \\
   &=& (g(a_{1,1},\ldots,a_{1,n}),\ldots,g(a_{s,1},\ldots,a_{s,n}))\not\in\sigma.
\end{eqnarray*}
iii) Let $c\in C_{\rho}\subseteq C_{\sigma}$ and
$\mbox{\boldmath{$c$}}_{1}=(c_{1,1},\ldots,c_{s,1}),\ldots,\mbox{\boldmath{$c$}}_{m}=(c_{1,m},\ldots,c_{s,m})\in
E_{k}^{s}$ satisfy the following condition: for every $1\leq
i_{1}<i_{2}<\cdots<i_{h}\leq s,
(\mbox{\boldmath{$b$}}_{i_{1}},\ldots,\mbox{\boldmath{$b$}}_{i_{h}})\not\in\rho$
where $\mbox{\boldmath{$b$}}_{j}=(c_{j,1},\ldots,c_{j,m})$ for
$1\leq j\leq s$, i.e., there is $1\leq j\leq m$ such that
$(c_{i_{1},j},\ldots,c_{i_{h},j})\notin\rho$. Note that this implies
that $\mbox{\boldmath{$b$}}_{1},\ldots,\mbox{\boldmath{$b$}}_{s}$
are pairwise distinct. For $1\leq i\leq n$, we consider the $m$-ary
function
$f^{i}_{\mbox{\boldmath{$c$}}_{1},\ldots,\mbox{\boldmath{$c$}}_{m}}$
defined by
\[f^{i}_{\mbox{\boldmath{$c$}}_{1},\ldots,\mbox{\boldmath{$c$}}_{m}}(\mbox{\boldmath{$x$}})=
\begin{cases}
a_{j,i} & \text{if }\mbox{\boldmath{$x$}}=\mbox{\boldmath{$b$}}_{j}\text{ for some }1\leq j\leq s, \\
c & \text{otherwise.}
\end{cases}
\]
It is easy to see that
$f^{i}_{\mbox{\boldmath{$c$}}_{1},\ldots,\mbox{\boldmath{$c$}}_{m}}\in
\Pol\{\rho,\sigma\}$  for $1\leq i\leq n$. Hence, as above
$f_{\mbox{\boldmath{$c$}}_{1},\ldots,\mbox{\boldmath{$c$}}_{m}}=g(f^{1}_{\mbox{\boldmath{$c$}}_{1},\ldots,\mbox{\boldmath{$c$}}_{m}},\ldots,f^{n}_{\mbox{\boldmath{$c$}}_{1},\ldots,\mbox{\boldmath{$c$}}_{m}})\in
\langle (\Pol\{\rho, \sigma\})\cup\{g\}\rangle$ and

$f_{\mbox{\boldmath{$c$}}_{1},\ldots,\mbox{\boldmath{$c$}}_{m}}(\mbox{\boldmath{$c$}}_{1},\ldots,\mbox{\boldmath{$c$}}_{m})=g(\mbox{\boldmath{$a$}}_{1},
\ldots,\mbox{\boldmath{$a$}}_{n})\not\in\sigma$.
\end{proof}
Now we can give the proof of Proposition \ref{prop 1}.
\begin{proof}[Proof (of Proposition \ref{prop 1})]
Let $g\in \Pol\rho\setminus \Pol\sigma$ be an $n$-ary operation. We
will show that $\Pol\rho=\langle (\Pol\{\rho,
\sigma\})\cup\{g\}\rangle$. We have $\langle (\Pol\{\rho,
\sigma\})\cup\{g\}\rangle\subseteq \Pol\rho$. Let $h\in \Pol\rho$ be
an $m$-ary operation, we will show that $h\in \langle (\Pol\{\rho,
\sigma\})\cup\{g\}\rangle$. Using Lemma \ref{lem 1}, we define the
set $S$ as follows:

If $\sigma$ is of type I or II, then $S=\{f_{\mbox{\boldmath{$c$}}}:
\mbox{\boldmath{$c$}}\in\sigma^{m}\}$;

If $\sigma\subsetneq\rho$ or $\sigma$ is of type IV, then

$
S=\{f_{\mbox{\boldmath{$c$}}_{1},\ldots,\mbox{\boldmath{$c$}}_{m}}:
\mbox{\boldmath{$c$}}_{1},\ldots, \mbox{\boldmath{$c$}}_{m}\text{
satisfy condition } (ii) \text{ of Lemma } \ref{lem 1}\}$;

If $\rho\subsetneq\sigma$ or $\sigma$ is of type V, then

$
S=\{f_{\mbox{\boldmath{$c$}}_{1},\ldots,\mbox{\boldmath{$c$}}_{m}}:
\mbox{\boldmath{$c$}}_{1},\ldots, \mbox{\boldmath{$c$}}_{m}\text{
satisfy condition } (iii) \text{ of Lemma } \ref{lem 1}\}$.

To simplify our
notations, we assume that $S=\{f_{i}:1\leq i\leq q\}$. We consider
the mapping $\ext: E_{k}^{m}\rightarrow E_{k}^{m+q}$ defined by
$\ext(\mbox{\boldmath{$x$}})=(\mbox{\boldmath{$x$}},f_{1}(\mbox{\boldmath{$x$}}),\ldots,f_{q}(\mbox{\boldmath{$x$}}))$.
We construct an $(m+q)$-ary function $\tilde{H}$ as follows:

If $\sigma$ is of type $l$ with $l\in\{I,III,IV,V\}$, then choosing
$c\in C_{\rho}\cap \sigma$ if $l=I$ and $c\in C_{\rho}\cap
C_{\sigma}$ otherwise, we define
\[\tilde{H}(\mbox{\boldmath{$y$}})=
\begin{cases}
h(\mbox{\boldmath{$x$}}) & \text{if there exists }\mbox{\boldmath{$x$}}\in E_{k}^{m}\text{ such that } \mbox{\boldmath{$y$}}=\ext(\mbox{\boldmath{$x$}}),\\
c & \mbox{otherwise.}
\end{cases}\]

If $\sigma$ is of type II, then for every maximal $\rho$-chain B, we
have $B\cap\sigma\neq \emptyset$. For all $\mbox{\boldmath{$y$}}\in
\sigma^{m+q}$, we set
$D_{\mbox{\boldmath{$y$}}}=\{h(\mbox{\boldmath{$x$}}):
\mbox{\boldmath{$x$}}\in E_{k}^{m}\wedge
(\ext(\mbox{\boldmath{$x$}}),\mbox{\boldmath{$y$}})\in \rho\}$. If
$h(\mbox{\boldmath{$x$}}),h(\mbox{\boldmath{$x$}}')\in
D_{\mbox{\boldmath{$y$}}}$, then
$(\ext(\mbox{\boldmath{$x$}}),\mbox{\boldmath{$y$}}),(\ext(\mbox{\boldmath{$x$}}'),\mbox{\boldmath{$y$}})\in\rho$,
hence
$(\ext(\mbox{\boldmath{$x$}}),\ext(\mbox{\boldmath{$x$}}'))\in\rho$
(due to $\sigma$ is of type II), so
$(\mbox{\boldmath{$x$}},\mbox{\boldmath{$x$}}')\in\rho$ and
$(h(\mbox{\boldmath{$x$}}),h(\mbox{\boldmath{$x$}}'))\in\rho$ (due
to $h\in \Pol\rho$); therefore $D_{\mbox{\boldmath{$y$}}}$ is a
$\rho$-chain;  there is a maximal $\rho$-chain B such that
$D_{\mbox{\boldmath{$y$}}}\subseteq B$.
 Set
 \[\eta=\{B: D_{\mbox{\boldmath{$y$}}}\subseteq B\text{ and B is a maximal }\rho\text{-chain}\}\]
  and $u_{\mbox{\boldmath{$y$}}}=
 \min[\sigma\cap(\underset{B\in\eta}{\cup}B)]$.
We set: \[\tilde{H}(\mbox{\boldmath{$y$}})=
\begin{cases}
h(\mbox{\boldmath{$x$}}) & \text{if }\exists
\mbox{\boldmath{$x$}}\in
E_{k}^{m},\mbox{\boldmath{$y$}}=\ext(\mbox{\boldmath{$x$}}),\\
u_{\mbox{\boldmath{$y$}}} & \text{if } \mbox{\boldmath{$y$}}\in\sigma^{m+q},\\
c & \text{otherwise.}
\end{cases}\]
where $c\in C_{\rho}$. We will show that $\tilde{H}\in
\Pol\{\rho,\sigma\}$.

Firstly, we show that $\tilde{H}\in \Pol\rho$.

Let
$\mbox{\boldmath{$a$}}_{1}=(a_{1,1},\ldots,a_{h,1})$,$\ldots$,$\mbox{\boldmath{$a$}}_{m+q}=(a_{1,m+q},\ldots,a_{h,m+q})\in\rho$
and set
$\mbox{\boldmath{$b$}}_{1}=(a_{1,1},\ldots,a_{1,m+q})$,$\ldots,$
$\mbox{\boldmath{$b$}}_{h}=(a_{h,1},\ldots,a_{h,m+q})$. If there is
$j\in\{1,\ldots,h\}$ such that
$\tilde{H}(\mbox{\boldmath{$b$}}_{j})=c$, then
$\tilde{H}(\mbox{\boldmath{$a$}}_{1},\ldots,\mbox{\boldmath{$a$}}_{m+q})=(\tilde{H}(\mbox{\boldmath{$b$}}_{1}),\ldots,\tilde{H}(\mbox{\boldmath{$b$}}_{h}))
\in\rho$. If for all $j\in\{1,\ldots,h\}$,
$\tilde{H}(\mbox{\boldmath{$b$}}_{j})\neq c$ and if $\sigma$ is not
of
 type II, then there exist $\mbox{\boldmath{$x$}}_{1},\ldots,\mbox{\boldmath{$x$}}_{h}$ such that
$\mbox{\boldmath{$b$}}_{1}=\ext(\mbox{\boldmath{$x$}}_{1}),\ldots,\mbox{\boldmath{$b$}}_{h}=\ext(\mbox{\boldmath{$x$}}_{h})$;
hence $\tilde{H}(\mbox{\boldmath{$a$}}_{1},\ldots,
\mbox{\boldmath{$a$}}_{m+q})=(h(\mbox{\boldmath{$x$}}_{1}),\ldots,h(\mbox{\boldmath{$x$}}_{h}))\in\rho$
(due to $h\in \Pol\rho$ and
$(\mbox{\boldmath{$x$}}_{1},\ldots,\mbox{\boldmath{$x$}}_{h})\in\rho$).

Now we suppose that $\sigma$ is of type II. Therefore $h=2$. If
there exist $\mbox{\boldmath{$x$}}_{1},\mbox{\boldmath{$x$}}_{2}\in
E_{k}^{m}$ such that
$\mbox{\boldmath{$b$}}_{1}=\ext(\mbox{\boldmath{$x$}}_{1})$ and
$\mbox{\boldmath{$b$}}_{2}=\ext(\mbox{\boldmath{$x$}}_{2})$, then
$\tilde{H}(\mbox{\boldmath{$a$}}_{1},
\mbox{\boldmath{$a$}}_{2})=(h(\mbox{\boldmath{$x$}}_{1}),h(\mbox{\boldmath{$x$}}_{2}))\in~\rho.$
 If $\mbox{\boldmath{$b$}}_{1},\mbox{\boldmath{$b$}}_{2}\in\sigma^{m+q}$, then $D_{\mbox{\boldmath{$b$}}_{1}}=D_{\mbox{\boldmath{$b$}}_{2}}$
 (due to $\rho$ being of type II,
 and $(\ext(\mbox{\boldmath{$x$}}),\mbox{\boldmath{$b$}}_{1})\in\rho$ if and only if $(\ext(x),\mbox{\boldmath{$b$}}_{2})\in\rho$ as
  $(\mbox{\boldmath{$b$}}_{1},\mbox{\boldmath{$b$}}_{2})\in\rho$);
 hence  $\tilde{H}(\mbox{\boldmath{$b$}}_{1})=\tilde{H}(\mbox{\boldmath{$b$}}_{2})$ and $\tilde{H}(\mbox{\boldmath{$a$}}_{1},\mbox{\boldmath{$a$}}_{2})\in\rho$.
  Otherwise, without loss of generality we suppose that
there exists $\mbox{\boldmath{$x$}}\in E_{k}^{m}$ such that
$\mbox{\boldmath{$b$}}_{1}=\ext(\mbox{\boldmath{$x$}})$ and
 $\mbox{\boldmath{$b$}}_{2}\in\sigma^{m+q}$. Therefore $h(\mbox{\boldmath{$x$}})\in D_{\mbox{\boldmath{$b$}}_{2}}$ and $(h(\mbox{\boldmath{$x$}}),u_{\mbox{\boldmath{$b$}}_{2}})\in\rho$.
 Hence $\tilde{H}(\mbox{\boldmath{$a$}}_{1},\mbox{\boldmath{$a$}}_{2})\in\rho$. We conclude that $\tilde{H}\in \Pol\rho$.

Secondly, we show that $\tilde{H}\in \Pol\sigma$.

i) We suppose that $\sigma$ is of type I or II. Let
$a_{1},\ldots,a_{m+q}\in\sigma$. By the construction of $\ext$, we
have $(a_{1},\ldots,a_{m+q})\not\in \ext(E_{k}^{m})$, hence
$\tilde{H}(a_{1},\ldots,a_{m+q})=c\in C_{\rho}\cap\sigma$ if
$\sigma$ is of type I or
$\tilde{H}(a_{1},\ldots,a_{m+q})=u_{(a_{1},\ldots,a_{m+q})}\in\sigma$
if $\sigma$ is of type II.

ii) We suppose that $\sigma$ is not of type I or II. Let
$\mbox{\boldmath{$a$}}_{1}=(a_{1,1},\ldots,a_{s,1}),\ldots,\mbox{\boldmath{$a$}}_{m+q}=(a_{1,m+q},\ldots,a_{s,m+q})\in\sigma$.
Set $\mbox{\boldmath{$b$}}_{1}=(a_{1,1},\ldots,a_{1,m+q}),\ldots,$
$\mbox{\boldmath{$b$}}_{s}=(a_{s,1},\ldots,a_{s,m+q})$. If there is
$j\in\{1,\ldots,~s\}$ such that
$\tilde{H}(\mbox{\boldmath{$b$}}_{j})=c$, then

$h(\mbox{\boldmath{$a$}}_{1},\ldots,\mbox{\boldmath{$a$}}_{m+q})=(\tilde{H}(\mbox{\boldmath{$b$}}_{1}),\ldots,\tilde{H}(\mbox{\boldmath{$b$}}_{s}))\in\sigma$
(due to $c\in C_{\sigma}$). If for all $1\leq j\leq s,
\tilde{H}(\mbox{\boldmath{$b$}}_{j})\neq c$, then there exist
$\mbox{\boldmath{$x$}}_{1},\ldots,\mbox{\boldmath{$x$}}_{s}\in
E_{k}^{m}$
such that $\textbf{\mbox{\boldmath{$b$}}}_{1}=\ext(\mbox{\boldmath{$x$}}_{1}), \mbox{\boldmath{$b$}}_{2}=\ext(\mbox{\boldmath{$x$}}_{2}),\ldots,\mbox{\boldmath{$b$}}_{s}=\ext(\mbox{\boldmath{$x$}}_{s})$.\\
If $\sigma\subsetneq\rho$ or $\sigma$ is of type IV, then
$\mbox{\boldmath{$x$}}_{1},\ldots,\mbox{\boldmath{$x$}}_{s}$ are
 not pairwise distinct elements of $E_{k}^{m}$(due to the construction of the set $S$).
  Therefore there is  $1\leq p<q\leq s$ such that $\mbox{\boldmath{$x$}}_{p}=\mbox{\boldmath{$x$}}_{q}$. Hence
$\tilde{H}(\mbox{\boldmath{$a$}}_{1},\ldots,\mbox{\boldmath{$a$}}_{m+q})=(\tilde{H}(\mbox{\boldmath{$b$}}_{1}),\ldots,
\tilde{H}(\mbox{\boldmath{$b$}}_{s}))=(h(\mbox{\boldmath{$x$}}_{1}),\ldots,
\ldots,h(\mbox{\boldmath{$x$}}_{s}))\in\sigma$ (due to $\sigma$
being totally reflexive).

 If $\rho\subsetneq\sigma$ or $\sigma$ is
of type V, then there exist $1\leq i_{1}<i_{2}<\ldots<i_{h}\leq s$
such that $(\mbox{\boldmath{$x$}}_{i_{1}},
\ldots,\mbox{\boldmath{$x$}}_{i_{h}})\in\rho$. Therefore
$\tilde{H}(\mbox{\boldmath{$a$}}_{1},\ldots,\mbox{\boldmath{$a$}}_{m+q})=(h(\mbox{\boldmath{$x$}}_{1}),\ldots,h(\mbox{\boldmath{$x$}}_{s}))$
and
$(h(\mbox{\boldmath{$x$}}_{i_{1}}),\ldots,h(\mbox{\boldmath{$x$}}_{i_{h}}))\in\rho$
(due to $h\in \Pol\rho$). So a permutation of
$(h(\mbox{\boldmath{$x$}}_{1}),\ldots,h(\mbox{\boldmath{$x$}}_{h}))$
 belongs to $\lambda\subset\sigma$, hence $\tilde{H}(\mbox{\boldmath{$a$}}_{1},\ldots,\mbox{\boldmath{$a$}}_{m+q})\in\sigma$
 (due to $\sigma$ being totally symmetric). We conclude that $\tilde{H}\in \Pol\sigma$. Therefore, for every
$\mbox{\boldmath{$x$}}\in E_{k}^{m},
h(\mbox{\boldmath{$x$}})=\tilde{H}(\mbox{\boldmath{$x$}},f_{1}(\mbox{\boldmath{$x$}}),\ldots,f_{q}(\mbox{\boldmath{$x$}}))$
and $h\in\langle (\Pol\{\rho, \sigma\})\cup\{g\}\rangle$.
\end{proof}

The more difficult part of this work is the completeness criterion
which will be discussed in the next section.

\section{Proof of the completeness criterion}
In this section, we will show that the relations of type I, II, III,
IV, and V are the only central  relations $\sigma$ such that
$\Pol\{\rho, \sigma\}$ is maximal in $\Pol\rho$. We recall that
$\rho$ is an $h$-ary central relation ($h\geq 2$) and $\sigma$ is an
$s$-ary
 central relation $(s\geq 1$). We will distinguish the following cases:

i) $s=1$,  ii) $s=h$,  iii) $h\leq s$, iv) $h\geq s$. We begin with the case  $s=1$.

\begin{proposition}\label{prop 2}
  Let $k\geq 3$, $\rho$ an h-ary central  relation $(h\geq 2)$ and
 $\sigma$ a unary central relation on $E_{k}$.
If $\Pol\{\rho, \sigma\}$ is maximal in $\Pol\rho$, then $\sigma$ is
of type I or II.
\end{proposition}
 The proof of Proposition \ref{prop 2} is spread across in the Lemmas \ref{lem 2}--\ref{lem 13}.
 For all $a\in E_{k}\setminus\sigma$ we denote by $c_{a}$ the unary constant
 operation on $E_{k}$ with value $a$. Firstly we suppose that $\rho$ is a  binary central relation.
Let $\tau$ be the unary relation defined on $E_{k}$  by
$$\tau=\{y\in E_{k}:\exists u\in\sigma, (u,y)\in\rho\}.$$
 Reflexivity of $\rho$ implies that $\sigma\subseteq\tau\subseteq E_{k}$
and we have the following three cases:

 (1) $\sigma=\tau$, (2)
$\sigma\subsetneq \tau\subsetneq E_{k}$, (3) $\tau=E_{k}$.
 \begin{lemma}\label{lem 2} Under the assumptions of Proposition \ref{prop
 2} and $\rho$ being binary, the subcase $\tau=\sigma$ is impossible.
 \end{lemma}

 \begin{proof}
 Let $c\in C_{\rho}$ and  $u\in\sigma$, then
 $(u,c)\in\rho$. So $c\in\tau=\sigma$; hence for all $a\in E_{k}, a\in\tau=\sigma$.
 Therefore, $E_{k}=\sigma$, which is a contradiction.
\end{proof}

 \begin{lemma}\label{lem 3}
 Under the assumptions of Proposition \ref{prop 2} and $\rho$ being binary, the subcase
 $\sigma\subsetneq\tau\subsetneq E_{k}$ is impossible.
 \end{lemma}

\begin{proof}
 We assume that  $\sigma\subsetneq\tau\subsetneq E_{k}$. Since $\tau\in[\{\rho, \sigma\}]$, we have $\Pol\{\rho, \sigma\}\subseteq
 \Pol\{\rho,\tau\}\subseteq \Pol\rho$. Let  $a\in\tau\setminus\sigma$; $c_{a}\not\in
\Pol\sigma$ and $c_{a}\in \Pol\{\rho, \tau\}$, therefore
$\Pol\{\rho, \sigma\}\subsetneq \Pol\{\rho, \tau\}$. Let $b\in
E_{k}\setminus\tau$; $c_{b}\not\in \Pol\tau$ and $c_{b}\in
\Pol\rho$; hence $\Pol\{\rho, \tau\}\subsetneq \Pol\rho$. Thus,
$\Pol\{\rho, \sigma\}\subsetneq \Pol\{\rho, \tau\}\subsetneq
\Pol\rho$. We obtain a contradiction with the maximality of
 $\Pol\{\rho, \sigma\}$ in $\Pol\rho$.
\end{proof}
We conclude that $\tau=E_{k}$, which implies that for all $x\in
E_{k}$ there exists $u\in\sigma$ such that $(u,x)\in\rho$. Let
$\gamma_{2}=\gamma$ be the binary relation defined before Theorem~
\ref{theo 1}. $\tau=E_{k}$ implies that $\gamma_{2}$ is reflexive;
$\gamma_{2}$ is symmetric by definition. Let $x\not\in\sigma$, there
is $u\in\sigma$ such that $(u,x)\in\rho$. So
$(u,x)\in\gamma_{2}\cap\rho$ and $\iota_{k}^{2}\subsetneq
\gamma_{2}\cap\rho\subseteq \rho$. Hence $\iota_{k}^{2}\subsetneq
\gamma_{2}\cap\rho\subsetneq\rho$ or $\gamma_{2}\cap\rho=\rho$.

\begin{lemma}\label{lem 4}
Under the assumptions of proposition \ref{prop 2} and $\rho$ being
binary, the subcase
$\iota_{k}^{2}\subsetneq\gamma_{2}\cap\rho\subsetneq \rho$ is
impossible.
\end{lemma}

\begin{proof}
We choose $a\in E_{k}\setminus\sigma$. $c_{a}\not\in \Pol\sigma$ and
$c_{a}\in \Pol\{\rho,(\gamma_{2}\cap\rho)\}$;
 therefore, $\Pol\{\rho,\sigma\}\subsetneq \Pol\{\rho,(\rho\cap\gamma_{2})\}$.
Let us consider an element $(e,d)\in\rho\setminus\gamma_{2}$ and
$(u,v)\in (\gamma_{2}\cap\rho)\setminus \iota_{k}^{2}$. The unary
operation $f$ defined on $E_{k}$ by
\[f(x)=
\begin{cases}
e & \text{if } x=u, \\
d & \text{otherwise.}
\end{cases}
\]
preserves $\rho$, but not $\gamma_{2}\cap\rho$ due to $(u,v)\in
\gamma_{2}\cap\rho$ and $(f(u),f(v))=(e,d)\not\in
\gamma_{2}\cap\rho$. Thus $\Pol\{\rho,\sigma\}\subsetneq
\Pol\{\rho,(\gamma_{2}\cap\rho)\}\subsetneq \Pol\rho$ contradicting
the maximality of $\Pol\{\rho, \sigma\}$ in $\Pol\rho$.
\end{proof}
From this lemma we have $\rho=\gamma_{2}\cap\rho$, i.e.
$\rho\subseteq\gamma_{2}$. The following three subcases are
possible:

(2.1) $\rho\subsetneq \gamma_{2}\subsetneq E_{k}^{2}$, (2.2)
$\gamma_{2}=E_{k}^{2}$ and (2.3) $\gamma_{2}=\rho$.

\begin{lemma}\label{lem 5}
Under the assumptions of Proposition \ref{prop 2} and $\rho$ being
binary, the subcase $\rho\subsetneq\gamma_{2}\subsetneq E_{k}^{2}$
is impossible.
\end{lemma}

\begin{proof}
We suppose that $\rho\subsetneq\gamma_{2}\subsetneq E_{k}^{2}$. From
$\rho\subsetneq \gamma_{2}\subsetneq E_{k}^{2}$, $\gamma_{2}$ is a
central relation of type III. By Proposition \ref{prop 1},
$\Pol\{\rho, \gamma_{2}\}$ is maximal in $\Pol\rho$. Hence
$\Pol\{\rho, \gamma_{2}\}\subsetneq \Pol\rho$. Let $a\in
E_{k}\setminus\sigma$, we have $c_{a}\in \Pol\{\rho, \gamma_{2}\}$
and $c_{a}\not\in \Pol\sigma$. Therefore $\Pol\{\rho,
\sigma\}\subsetneq \Pol\{\rho, \gamma_{2}\}\subsetneq \Pol\rho$
contradicting the fact that $\Pol\{\rho, \sigma\}$ is a submaximal
clone of $\Pol\rho$.
\end{proof}

Hence, we are left with cases (2.2) and (2.3). First we suppose that
$E_{k}^{2}=\gamma_{2}$ and we set for all $2\leq t\leq k$
 $$\gamma_{t}=\{(a_{1},\ldots,a_{t})\in E_{k}^{t}: \exists
u\in\sigma, \{(a_{1},u),\ldots,(a_{t},u)\}\subseteq\rho\}.$$
 Assuming that $\gamma_{k}\neq E_{k}^{k}$, let $n$ be the least integer such that
 $\gamma_{n}\neq E_{k}^{n}$. Then
$\gamma_{n-1}=E_{k}^{n-1}$ and $n >2$.

\begin{lemma}\label{lem 6}
Under the assumption of Proposition \ref{prop 2}, $\rho$ being
binary and
 $\gamma_{2}=E_{k}^{2}$, the subcase  $\gamma_{k}\neq E_{k}^{k}$ is impossible.
\end{lemma}

\begin{proof}
 $\gamma_{n}$ is totally symmetric by definition. Let  $a_{1},\ldots, a_{n-1}\in E_{k}$. We will show that
$(a_{1},\ldots,a_{n-1},a_{n-1})\in\gamma_{n}$.
$(a_{1},\ldots,a_{n-1})\in E_{k}^{n-1}=\gamma_{n-1}$, thus there
exists $u\in\sigma$ such that
$\{(a_{1},u),\ldots,(a_{n-1},u)\}\subseteq\rho$. We deduce that
$(a_{1},\ldots,a_{n-1},a_{n-1})\in\gamma_{n}$  and then total
symmetry of
 $\gamma_{n}$ implies that $\gamma_{n}$ is totally reflexive.

 Since $\gamma_{n}\in [\{\rho, \sigma\}]$, we have $\Pol\{\rho, \sigma\}\subseteq
\Pol\{\rho, \gamma_{n}\}\subseteq \Pol\rho$. Let $a\in
E_{k}\setminus\sigma$; $c_{a}\in \Pol\{\rho, \gamma_{n}\}$ and
$c_{a}\not\in \Pol\sigma$. Let $(a,b)\in E_{k}^{2}\setminus\rho$. We
set
$\mbox{\boldmath{$w$}}_{1}=(b,a,\ldots,a),\ldots,\mbox{\boldmath{$w$}}_{n}=(a,a,\ldots,a,b)$.
Let $c\in C_{\rho}$, $(u_{1},\ldots,u_{n})\in
E_{k}^{n}\setminus\gamma_{n}$. Since
$\mbox{\boldmath{$w$}}_{1},\ldots,\mbox{\boldmath{$w$}}_{n}$ are
pairwise distinct, the following $n$-ary operation $f$ on $E_{k}$ is
well defined.
\[
f(\mbox{\boldmath{$x$}})=
\begin{cases}
u_{i} & \text{if }\mbox{\boldmath{$x$}}=\mbox{\boldmath{$w$}}_{i}\text{ for some }1\leq i\leq n, \\
c & \text{otherwise}.
\end{cases}
\]
$\{\mbox{\boldmath{$w$}}_{1},\ldots,\mbox{\boldmath{$w$}}_{n}\}\subseteq\gamma_{n}$
and
$f(\mbox{\boldmath{$w$}}_{1},\ldots,\mbox{\boldmath{$w$}}_{n})=(f(\mbox{\boldmath{$w$}}_{1}),\ldots,f(\mbox{\boldmath{$w$}}_{n}))=(u_{1},\ldots,u_{n})\not\in\gamma_{n}$.\\
So $f\not\in \Pol\gamma_{n}$. Let
$\mbox{\boldmath{$x$}}=(a_{1},\ldots,a_{n})$ and
$\mbox{\boldmath{$y$}}=(b_{1},\ldots,b_{n})$ such that
$(\mbox{\boldmath{$x$}},\mbox{\boldmath{$y$}})\in\rho$. We will show
that $(f(\mbox{\boldmath{$x$}}),f(\mbox{\boldmath{$y$}}))\in\rho$.
By the construction of $\mbox{\boldmath{$w$}}_{i}$,
$(f(\mbox{\boldmath{$x$}}),f(\mbox{\boldmath{$y$}}))\in\{(u_{i},u_{i}),(u_{i},c),(c,c),(c,u_{i})\colon
 1\leq i\leq n\}\subseteq\rho$ (due to  $(\mbox{\boldmath{$w$}}_{i},\mbox{\boldmath{$w$}}_{j})\not\in\rho$ for all $i \neq j$).
Hence  $f\in \Pol\rho$ and $\Pol\{\rho, \sigma\}\subsetneq
\Pol\{\rho, \gamma_{n}\}\subsetneq \Pol\rho$. Therefore $\Pol\{\rho,
\sigma\}$ is not a submaximal clone of $\Pol\rho$.
\end{proof}

Hence, in case (2.2) we have $\gamma_{k}=E_{k}^{k}$.
\begin{lemma}\label{lem 7}
Under the assumptions of Proposition \ref{prop 2}, $\rho$ being
binary and $\gamma_{k}=E_{k}^{k}$, $\sigma$ is of type I.
\end{lemma}

\begin{proof}
We have $\{0,1,\ldots,k-1\}^{k}=E_{k}^{k}=\gamma_{k}$; hence there exists $u\in\sigma$ such that\\
 $\{(0,u),\ldots,(k-1,u)\}\subseteq\rho$. Thus $u\in C_{\rho}$ and $\sigma$ is of type I.
\end{proof}
Now we study the subcase (2.3) $\gamma_{2}=\rho$. Let $\Gamma=\{
B\subseteq E_{k}\colon B^{2}\subseteq\rho\}$ and $m=\max\{
\Card(B)\colon B\in \Gamma\}$. We have  $m\geq 2$; for all
$l\in\{2,3,\ldots,m\}$ we set:
$$\rho_{l}=\{(a_{1},\ldots,a_{l})\in E_{k}^{l}:
\{a_{1},\ldots,a_{l}\}^{2}\subseteq\rho\}.$$ Since
$\gamma_{2}=\rho_{2}=\rho$, we have $\gamma_{m}\subseteq \rho_{m}$.
\begin{lemma}\label{lem 8}
Under the assumptions of Proposition \ref{prop 2} and $\rho$ being
binary, the case $\gamma_{m}\subsetneq\rho_{m}$ is impossible.
\end{lemma}

\begin{proof}
It is similar to the proof of Lemma~\ref{lem 6}. Choose
$(a,b)\in\rho\setminus\tau_{k}^{2}$ and
$(u_{1},\ldots,u_{n})\in\rho_{n}\setminus\gamma_{n}$ for the least
$n$ such that $\gamma_{n}\varsubsetneq\rho_{n}$. The function $f$
defined in the proof of Lemma \ref{lem 6} preserves $\rho$, but not
$\gamma_{n}$ because $\{u_{1},\ldots,u_{n}\}$ is a $\rho$-chain,
$(\omega_{1},\ldots,\omega_{n})\in~\gamma_{n}$, and
$(f(\omega_{1}),\ldots,f(\omega_{n}))=(u_{1},\ldots,u_{n})\notin\gamma_{n}$
. Hence $\Pol\{\rho, \gamma_{n}\}\varsubsetneq\Pol\rho$. Therefore
$\Pol\{\rho, \sigma\}\varsubsetneq\Pol\{\rho,
\gamma_{n}\}\varsubsetneq\Pol\rho$, contradicting the submaximality
of $\Pol\{\rho, \sigma\}$.
\end{proof}
From Lemma \ref{lem 8}, we conclude that $\gamma_{m}=\rho_{m}$.
\begin{lemma}\label{lem 9}
Under the assumptions of Proposition \ref{prop 2}, $\rho=\gamma_{2}$
and $\gamma_{m}=\rho_{m}$, we have that for all maximal
$\rho$-chains $B$, $B\cap\sigma\neq \emptyset$.
\end{lemma}
\begin{proof}
Let $B=\{a_{1},\ldots,a_{n}\}$ be a maximal $\rho$-chain with $n\leq
m$ elements. By duplicating entries, we can find
$u_{1},\ldots,u_{m}$ such that
$(u_{1},\ldots,u_{m})\in\rho_{m}=\gamma_{m}$ and
$\{u_{1},\ldots,u_{m}\}=\{a_{1},\ldots,a_{n}\}$. Consequently, there
exists $u\in\sigma$ such that
$\{(u_{1},u),\ldots,(u_{m},u)\}\subseteq\rho$; hence $
(a_{1},u),\ldots,(a_{n},u)\in\rho$. Therefore $B\cup\{u\}$ is a
$\rho$-chain containing $B$. By maximality,  $u\in B$ and
$B\cap\sigma\neq \emptyset$. We conclude that $\sigma$ is of type
II.
\end{proof}
Secondly we suppose that $\rho$ is a $h$-ary central relation with $3\leq h$. We distinguish two subcases:

(1) $C_{\rho}\cap\sigma\neq \emptyset$ or (2) $C_{\rho}\cap\sigma=\emptyset$.

The first one is easy; the second one we shall prove to be in
contradiction with $h>2$.

\begin{lemma}\label{lem 10}
Under the assumptions of Proposition \ref{prop 2}, $h>2$ and
$C_{\rho}\cap\sigma\neq\emptyset$, we obtain a relation of type I.
\end{lemma}
\begin{proof}
$\sigma$ obviously satisfies the condition I of Theorem \ref{theo
1}.
\end{proof}
Now, we suppose that $C_{\rho}\cap\sigma=\emptyset$. Let
$\omega\in\sigma$, there exist $a_{2},\ldots,a_{h}\in E_{k}$
 such that $(\omega,a_{2},\ldots,a_{h})\not\in\rho$. For $n\geq h-1$, we set:
\[\alpha_{n}=\{(b_{1},\ldots,b_{n})\in E_{k}^{n}:\exists
u\in\sigma\forall 1\leq i_{1} <i_{2}<\ldots < i_{h-1}\leq n,
(u,b_{i_{1}},\ldots,b_{i_{h-1}})\in\rho \}.\] For $n\geq h-1$,
$\alpha_{n}$ is totally symmetric. $\alpha_{h-1}$ is totally
reflexive and
$(\omega,a_{2},\ldots,a_{h-1})\in\alpha_{h-1}\setminus\iota_{k}^{h-1}$.
Therefore $\iota_{k}^{h-1}\subsetneq\alpha_{h-1}\subsetneq
E_{k}^{h-1}$ or $\alpha_{h-1}=E_{k}^{h-1}$.

\begin{lemma}\label{lem 11}
 Under the assumptions of Proposition \ref{prop 2}, $h>2$ and $C_{\rho}\cap\sigma=\emptyset$, the subcase
  $\iota_{k}^{h-1}\subsetneq\alpha_{h-1}\subsetneq E_{k}^{h-1}$ is impossible.
\end{lemma}

\begin{proof}
It is easy to check that $\Pol\{\rho, \sigma\}\subsetneq \Pol\{\rho,
\alpha_{h-1}\}\subseteq \Pol\rho$. Let $(a_{2},\ldots,a_{h})\in
E_{k}^{h-1}\setminus\alpha_{h-1},
(u_{2},\ldots,u_{h})\in\alpha_{h-1}\setminus\iota_{k}^{h-1}$ and
define the unary operation $f$ on $E_{k}$ by
\[
f(x)=
\begin{cases}
a_{i} & \text{if }x=u_{i}\text{ for some }2\leq i\leq h,\\
a_{2} & \text{otherwise.}
\end{cases}
\]
$f\in \Pol\rho$ by total reflexivity of $\rho$.
$(u_{2},\ldots,u_{h})\in\alpha_{h-1}$
 and $(f(u_{2}),\ldots,f(u_{h}))=(a_{2},\ldots,a_{h})\not\in\alpha_{h-1}$. So $f\in \Pol\rho\setminus \Pol\alpha_{h-1}$.
Thus $\Pol\{\rho,\sigma\}\subsetneq \Pol\{\rho,
\alpha_{h-1}\}\subsetneq \Pol\rho$; contradicting the maximality of
$\Pol\{\rho, \sigma\}$ in $\Pol\rho$.
\end{proof}
We continue with the subcase $\alpha_{h-1}=E_{k}^{h-1}$.
 It is easy to see that for all $j$ such that $h-1\leq j\leq n$, if $\alpha_{j}=E_{k}^{j}$, then
$\alpha_{j+1}$ is totally reflexive. Since
$\alpha_{h-1}=E_{k}^{h-1}$, there exists $u\in\sigma$ such that
$(u,a_{2},\ldots,a_{h})\in\rho$. We obtain the following two
subcases:\\
\textbf{Case 1:} for all $j\in \{2,\ldots,h\}, (\omega,a_{2},\ldots,a_{j-1},u,a_{j+1},\ldots,a_{h})\in\rho$;\\
\textbf{Case 2:} there exists $j\in\{2,\ldots,h\}$ such that
 $(\omega,a_{2},\ldots,a_{j-1},u,a_{j+1},\ldots,a_{h})\not\in\rho$.

We will study these two cases in the following two lemmas. Firstly
we study \textbf{Case~ 1}. We suppose that for all $j\in
\{2,\ldots,h\},
(\omega,a_{2},\ldots,a_{j-1},u,a_{j+1},\ldots,a_{h})\in~\rho$.

\begin{lemma}\label{lem 12}
Under the assumptions of Proposition \ref{prop 2}, $h> 2$ and
$C_{\rho}\cap\sigma=\emptyset$, \textbf{Case 1}, i.e., there exist
$\omega, u\in\sigma, a_{2},\ldots,a_{h}\in E_{k}$ such that
$(\omega,a_{2},\ldots,a_{h})\not\in\rho,
(u,a_{2},\ldots,a_{h})\in\rho$ and for all $j\in\{2,\ldots,h\},
(\omega,a_{2},\ldots,a_{j-1},u,a_{j+1},\ldots,a_{h})\in\rho$, is
impossible.
\end{lemma}

\begin{proof}
 Since $ C_{\rho}\cap\sigma=\emptyset$ we have $\alpha_{k}\neq E_{k}^{k}$. Let $n_{0}\geq h-1$ be the least integer such that
 $\alpha_{n_{0}}\neq E_{k}^{n_{0}}$.
 Since $\alpha_{h-1}=E_{k}^{h-1}$, we conclude that $n_{0}\geq h$ and $\alpha_{n_{0}}$ is totally reflexive and totally symmetric.
 We will show that $\Pol\{\rho, \sigma\}\subsetneq \Pol\{\rho, \alpha_{n_{0}}\}\subsetneq \Pol\rho$.

Since $\alpha_{n_{0}}\in[\{\rho, \sigma\}]$, we have $\Pol\{\rho,
\sigma\}\subsetneq \Pol\{\rho, \alpha_{n_{0}}\}$. Let
$(b_{1},\ldots,b_{h}):=(\omega,a_{2},\ldots,a_{h})\in
E_{k}^{h}\setminus\rho$, $W=\{(i_{1},\ldots,i_{h}): 1\leq
i_{1}<i_{2}<\cdots<i_{h}\leq n_{0}\}$. We set
$W=\{(i^{j}_{1},\ldots,i_{h}^{j}):1\leq j\leq q\}$. For all $1\leq
j\leq q$ we set
$\mbox{\boldmath{$y$}}_{j}=(x_{j,1},\ldots,x_{j,n_{0}})$ where
\[
x_{j,p}=
\begin{cases}
b_{l} & \text{if }p=i_{l}^{j} \text{ for some }1\leq l\leq h, \\
b_{1} & \text{otherwise.}
\end{cases}
\]
Let us set $\mbox{\boldmath{$x$}}_{i}=(x_{1,i},\ldots,x_{q,i}),1\leq
i\leq n_{0}$, and choose $(v_{1},\ldots,v_{n_{0}})\in
E_{k}^{n_{0}}\setminus\alpha_{n_{0}}$ and $c\in C_{\rho}$. Note that
$\mbox{\boldmath{$x$}}_{i}\neq\mbox{\boldmath{$x$}}_{l}$ for $1\leq
i< l\leq n_{0}$ by the construction of
$\mbox{\boldmath{$y$}}_{1},\ldots,\mbox{\boldmath{$y$}}_{q}$, so we
can define the $q$-ary operation $f$ on $E_{k}$ by
\[
f(\mbox{\boldmath{$x$}})=
\begin{cases}
v_{i} & \text{if }\mbox{\boldmath{$x$}}=\mbox{\boldmath{$x$}}_{i}\text{ for some }1\leq i\leq n_{0}, \\
c & \text{otherwise.}
\end{cases}
\]
Because of the choice of $(b_{1},\ldots,b_{n})$, the construction of
$\mbox{\boldmath{$y$}}_{1},\ldots,\mbox{\boldmath{$y$}}_{q}$, and of
the assumptions of \textbf{Case} 1, we have
$\{\mbox{\boldmath{$y$}}_{1},\ldots,\mbox{\boldmath{$y$}}_{q}
\}\subseteq\alpha_{n_{0}}$; moreover
$f(\mbox{\boldmath{$y$}}_{1},\ldots,\mbox{\boldmath{$y$}}_{q})=(f(\mbox{\boldmath{$x$}}_{1}),\ldots,
f(\mbox{\boldmath{$x$}}_{n_{0}}))=(v_{1},\ldots,v_{n_{0}})\not\in\alpha_{n_{0}}$.
So $f\not\in \Pol\alpha_{n_{0}}$. Using the construction of
$\mbox{\boldmath{$y$}}_{j},1\leq j\leq q$, and the fact that $c\in
C_{\rho}$ and $(b_{1},\ldots,b_{n})\notin\rho$, we can show that
$f\in \Pol\rho$, so $\Pol\{\rho, \alpha_{n_{0}}\}\subsetneq
\Pol\rho$. Therefore $\Pol\{\rho, \sigma\}\subsetneq \Pol\{\rho,
\alpha_{n_{0}}\}\subsetneq \Pol\rho$; contradicting the maximality
of $\Pol\{\rho, \sigma\}$ in $\Pol\rho$.
\end{proof}
Secondly, we discuss \textbf{Case} 2: there exists
$j\in\{2,\ldots,h\}$ such that
$(\omega,a_{2},\ldots,a_{j-1},$\\
$u,a_{j+1},\ldots,a_{h})\notin\rho$. Without loss of generality, we
suppose that $(\omega,u,a_{3},\ldots,a_{h})\not\in\rho$. Let
\[F=\{\{b_{1},\ldots,b_{h-1}\}\subseteq
E_{k}:\Card(\{b_{1},\ldots,b_{h-1}\})=h-1\text{ and
}C_{\rho}\cap\{b_{1},\ldots,b_{h-1}\}=\emptyset\}\] and
$m=\Card(F)$. For $1\leq j\leq m$ set
\[\beta_{j(h-1)+1}=\{(b_{1},\ldots,b_{j(h-1)+1})\in
E_{k}^{j(h-1)+1}:\exists u\in\sigma: \{(u,b_{2},\ldots,b_{h}),\]
\[(u,b_{h+1},\ldots,b_{2h-1}),\ldots,
(u,b_{(j-1)(h-1)+2},\ldots,b_{j(h-1)+1})\}\subseteq\rho\}.\]

\begin{lemma}\label{lem 13}
Under the assumptions of Proposition \ref{prop 2}, $h> 2$ and
$C_{\rho}\cap\sigma=\emptyset$, \textbf{Case 2}, i.e., there exist
$\omega,u \in\sigma, a_{2},\ldots,a_{h}\in E_{k}$ such that
$(\omega,a_{2},\ldots,a_{h})\not\in\rho$,
$(u,a_{2},\ldots,a_{h})\in\rho$,
$(\omega,u,a_{3},\ldots,a_{h})\notin\rho$, is impossible.
\end{lemma}

\begin{proof}
 Since $C_{\rho}\cap\sigma=\emptyset$, we have $\beta_{m(h-1)+1}\neq E_{k}^{m(h-1)+1}$. Let $m_{0}\geq 1$ be the least integer such that
$\beta_{m_{0}(h-1)+1}\neq E_{k}^{m_{0}(h-1)+1}$. We will show that
$\Pol\{\rho, \sigma\}\subsetneq \Pol\{\rho,
\beta_{m_{0}(h-1)+1}\}\subsetneq \Pol\rho$.

Since $\beta_{m_{0}(h-1)+1}\in[\{\rho, \sigma\}]$, we have
$\Pol\{\rho, \sigma\}\subsetneq \Pol\{\rho,
\beta_{m_{0}(h-1)+1}\}\subseteq \Pol\rho$.
 Since $\beta_{m_{0}(h-1)+1}\neq E_{k}^{m_{0}(h-1)+1}$, there exists $(v_{1},\ldots,v_{m_{0}(h-1)+1})\in E_{k}^{m_{0}(h-1)+1}\setminus\beta_{m_{0}(h-1)+1}$.
 Let $W'=\{(i_{1},\ldots,i_{h}): 1\leq i_{1}<i_{2}<\ldots<i_{h}\leq
m_{0}(h-1)+1\}$. For the reason of simpler notations we set\\
$W'=\{(i_{1}^{j},
  \ldots,i_{h}^{j}):1\leq j\leq q\}$. Set $\mbox{\boldmath{$y$}}_{j}=(x_{j,1},\ldots,x_{j,m_{0}(h-1)+1})$ with
 \[
x_{j,p}=
\begin{cases}
\omega & \text{if }p=i_{1}^{j}, \\
u & \text{if }p=i_{2}^{j}, \\
a_{l} & \text{if }p=i_{l}^{j}\text{ for some }3\leq l\leq h, \\
u & \text{otherwise.}
\end{cases}
\]
Set $\mbox{\boldmath{$x$}}_{i}=(x_{1,i},\ldots,x_{q,i}),1\leq i\leq m_{0}(h-1)+1$. We have $(\mbox{\boldmath{$x$}}_{i_{1}},\ldots,\mbox{\boldmath{$x$}}_{i_{h}})\not\in\rho$ for all\\
$(1\leq i_{1}<i_{2}<\ldots<i_{h}\leq m_{0}(h-1)+1$ ($\ast$)  (due to
$(\omega,u,a_{3},\ldots,a_{h})\not\in\rho$). Choose $c\in C_{\rho}$
and consider the $q$-ary operation $f$ defined on $E_{k}$ by
\[
f(\mbox{\boldmath{$x$}})=
\begin{cases}
v_{i} & \text{if }\mbox{\boldmath{$x$}}=\mbox{\boldmath{$x$}}_{i}\text{ for some }1\leq i\leq m_{0}(h-1)+1, \\
c & \text{otherwise.}
\end{cases}
\]
This operation is well defined, similarly as in Lemma~\ref{lem 12},
and we have $f\in \Pol\rho$ (due to $c\in C_{\rho}$ and ($\ast$)).
By the construction
$\{\mbox{\boldmath{$y$}}_{1},\ldots,\mbox{\boldmath{$y$}}_{q}\}\subseteq\beta_{m_{0}(h-1)+1}$,
furthermore
$(f(\mbox{\boldmath{$x$}}_{1}),\ldots,f(\mbox{\boldmath{$x$}}_{m_{0}(h-1)+1}))=(v_{1},\ldots,v_{m_{0}(h-1)+1})\not\in
\beta_{m_{0}(h-1)+1}$, so $f\not\in \Pol\beta_{m_{0}(h-1)+1}$.
Therefore $\Pol\{\rho, \sigma\}\subsetneq \Pol\{\rho,
\beta_{m_{0}(h-1)+1}\}\subsetneq \Pol\rho$; contradicting the
maximality of $\Pol\{\rho, \sigma\}$ in $\Pol\rho$.
\end{proof}
\begin{proof}[Proof (of Proposition \ref{prop 2})]
It follows from Lemmas \ref{lem 2}--\ref{lem 13}.
\end{proof}

We continue with the case $s=h$.

\begin{proposition}\label{prop 3}
Let $k\geq 3$, $\rho$ and $\sigma$ two h-ary central relations  on
$E_{k}$ ($h\geq 2)$. If $\Pol\{\rho, \sigma\}$ is maximal in
$\Pol\rho$, then $\sigma$ is of type III.
\end{proposition}
The proof of Proposition \ref{prop 3} is contained in Lemmas
\ref{lem 14}--\ref{lem 16}. We set $\gamma=\rho\cap\sigma$,
$\gamma_{1}=\{(x_{1},\ldots,x_{h})\in E_{k}^{h}: \exists u\in E_{k},
\forall 1\leq i\leq h,
(x_{1},\ldots,x_{i-1},u,x_{i+1},\ldots,x_{h})\in\gamma\}$ and
$\beta=\rho\cap\gamma_{1}$. By the definition of $\gamma_{1}$, we
have $\gamma\subseteq\gamma_{1}$, $\gamma_{1}$ reflexive or totally
reflexive, and $\gamma_{1}$ symmetric or totally symmetric. Since
$\gamma=\rho\cap\sigma$, we have two cases: (3.1)
$\gamma\in\{\rho,\sigma\}$ and (3.2) $\gamma\varsubsetneq\rho$ and
$\gamma\varsubsetneq\sigma$.

\begin{lemma}\label{lem 14}
Under the assumptions of Proposition \ref{prop 3} and
$\gamma\in\{\rho,\sigma\}$, we obtain a relation of type III.
\end{lemma}

\begin{proof}
 Since $\rho\cap\sigma\in\{\rho,\sigma\}$, we have $\rho\cap\sigma=\rho$ or $\rho\cap\sigma=\sigma$; hence, $\rho\subsetneq\sigma$ or $\sigma\subsetneq\rho$.
 Therefore, $\sigma$ is of type III.
\end{proof}

We look at the case $\gamma\varsubsetneq\rho$ and
$\gamma\varsubsetneq \sigma$, i.e $\rho\nsubseteq\sigma$ and
$\sigma\nsubseteq\rho$. We have the following subcases: (i)
$\gamma_{1}\cap\rho=\gamma$ and (ii)
$\gamma\varsubsetneq\gamma_{1}\cap\rho$.

We begin with $\gamma_{1}\cap\rho=\gamma$. Hence, $\gamma$ is not a
central relation because $\gamma_{1}\neq E_{k}^{h}$. Choose $c\in
C_{\rho}$, $\omega\in C_{\sigma}$, and
$a_{2},\ldots,a_{h},b_{2},\ldots,b_{h}$ such that
$(c_{1},\ldots,c_{h}):=(c,a_{2},\ldots,a_{h})\notin\sigma$ and
$(\omega_{1},\ldots,\omega_{h}):=(\omega,b_{2},\ldots,b_{h})\notin\rho$.

\begin{lemma}\label{lem 15}
Under the assumptions of Proposition \ref{prop 3} and
$\rho\cap\sigma\not\in\{\rho,\sigma\}$, the subcase
 $\gamma=\gamma_{1}\cap\rho$ is impossible.
\end{lemma}

\begin{proof}
We consider the $t$-ary ($t\geq h$) relation $\beta_{t}$ defined on
$E_{k}$ by
\[\beta_{t}=\{(x_{1},\ldots,x_{t})\in E_{k}^{t}: \exists
u_{x_{1}}\in E_{k}:\,
\forall\{i_{1},\cdots,i_{h-1}\}\subseteq\{1,\ldots, t\},\]
\[(x_{i_{1}},\ldots,x_{i_{h-2}},x_{1},u_{x_{1}})\in~\gamma\text{ and
}(x_{i_{1}},\ldots,x_{i_{h-1}},u_{x_{1}})~\in\rho\}.\] Taking
$u_{\omega}=c$, we see that
$(\omega_{1},\ldots,\omega_{h})=(\omega,b_{2},\ldots,b_{h})\in\beta_{h}\setminus\rho$.
Hence, $\rho\varsubsetneq\beta_{h}$. Therefore we have two subcases:
$\beta_{h}\varsubsetneq E_{k}^{h}$ or $\beta_{h}=E_{k}^{h}$.

1) First we suppose that $\beta_{h}\varsubsetneq E_{k}^{h}$. We will
show that
$\Pol\{\rho,\sigma\}\varsubsetneq\Pol\{\rho,\beta_{h}\}\varsubsetneq\Pol\rho$.
Since $\beta_{h}\in[\{\rho,\sigma\}]$, we have
$\Pol\{\rho,\sigma\}\subseteq\Pol\{\rho,\beta_{h}\}\subseteq\Pol\rho$.
The unary operation $f$ defined on $E_{k}$ by
\[
f(x)=
\begin{cases}
c_{1} & \text{if }x=\omega_{1}=\omega,\\
c_{i} & \text{if } x=c_{i}\text{ for some }2\leq i\leq h,\\
c & \text{otherwise.}
\end{cases}
\]
belongs to $\Pol\{\rho,\beta_{h}\}$(due to $\rho\subseteq\beta_{h}$,
$\rho$ totally reflexive and
$\img(f)=\{c_{1}=c,c_{2},\ldots,c_{h}\}$), but not to $\Pol\sigma$
because $(\omega,c_{2},c_{3},\ldots,c_{h})\in\sigma$ and
$f(\omega,c_{2},c_{3},\ldots,c_{h})=(c_{1},\ldots,c_{h})\notin\sigma$;
hence, $\Pol\{\rho,\sigma\}\varsubsetneq\Pol\{\rho,\beta_{h}\}$. It
remains to show that $\Pol\{\rho,\beta_{h}\}\subsetneq\Pol\rho$.
Moreover, $\rho$ and $\beta_{h}$ are two different central relations
(due to $\rho\varsubsetneq\beta_{h}$); therefore, $\Pol\rho$ and
$\Pol\beta_{h}$ are two different maximal clones and
$\Pol\{\rho,\beta_{h}\}\varsubsetneq\Pol\rho$. Hence,
$\Pol\{\rho,\sigma\}$ is not a submaximal clone of $\Pol\rho$,
contradicting the assumption of the lemma.

2) Now we suppose that $\beta_{h}=E_{k}^{h}$. It yields two
possibilities: $\beta_{k}\neq E_{k}^{k}$ or $\beta_{k}=E_{k}^{k}$.

a) First we suppose that $\beta_{k}\neq E_{k}^{k}$. Let $m_{0}>h$ be
the least integer such that $\beta_{m_{0}}\neq E_{k}^{m_{0}}$. We
will show that $\Pol\{\rho, \sigma\}\subsetneq \Pol\{\rho,
\beta_{m_{0}}\}\subsetneq \Pol\rho$. Since $\beta_{m_{0}}\in[\{\rho,
\sigma\}]$, we have $\Pol\{\rho, \sigma\}\subseteq \Pol\{\rho,
\beta_{m_{0}}\}\subseteq \Pol\rho$. The unary operation $f$ defined
above preserves $\rho$ and $\beta_{m_{0}}$(due to $m_{0}>h$ and
$\beta_{m_{0}}$ totally reflexive), and it does not preserve
$\sigma$. Thus
$\Pol\{\rho,\sigma\}\varsubsetneq\Pol\{\rho,\beta_{m_{0}}\}$. Since
$\beta_{m_{0}}\neq E_{k}^{m_{0}}$, there exists
$(v_{1},\ldots,v_{m_{0}})\in E_{k}^{m_{0}}\setminus\beta_{m_{0}}$.
Let $W'=\{(i_{1},\ldots,i_{h}): 1\leq i_{1}<i_{2}<\ldots<i_{h}\leq
m_{0}\}$. For the reason of simpler notations we set
$W'=\{(i_{1}^{j},
  \ldots,i_{h}^{j}):1\leq j\leq q\}$. Set $\mbox{\boldmath{$y$}}_{j}=(x_{j,1},\ldots,x_{j,m_{0}})$ with
 \[
x_{j,p}=
\begin{cases}
\omega_{l} & \text{if }p=i_{l}^{j}\text{ for some }1\leq l\leq h, \\
\omega_{1} & \text{otherwise.}
\end{cases}
\]
Set $\mbox{\boldmath{$x$}}_{i}=(x_{1,i},\ldots,x_{q,i}),1\leq i\leq
m_{0}$. We have $(\mbox{\boldmath{$x$}}_{i_{1}},
\ldots,\mbox{\boldmath{$x$}}_{i_{h}})\not\in\rho$ for all $(1\leq
i_{1}<i_{2}<\ldots<i_{h}\leq m_{0}$ ($\ast$)  (due to
$(\omega_{1},\ldots,\omega_{h})\not\in\rho$). Consider the $q$-ary
operation $f$ defined on $E_{k}$ by
\[
f(\mbox{\boldmath{$x$}})=
\begin{cases}
v_{i} & \text{if }\mbox{\boldmath{$x$}}=\mbox{\boldmath{$x$}}_{i}\text{ for some }1\leq i\leq m_{0}, \\
c & \text{otherwise.}
\end{cases}
\]
This operation is well defined, similarly as in Lemma ~\ref{lem 12},
and we have $f\in \Pol\rho$ (due to $c\in C_{\rho}$ and ($\ast$)).
Furthermore,
$\{\mbox{\boldmath{$y$}}_{1},\ldots,\mbox{\boldmath{$y$}}_{q}\}\subseteq\beta_{m_{0}}$,
but
$(f(\mbox{\boldmath{$x$}}_{1}),\ldots,f(\mbox{\boldmath{$x$}}_{m_{0}}))=(v_{1},\ldots,v_{m_{0}})\not\in
\beta_{m_{0}}$. So, $f\not\in \Pol\beta_{m_{0}}$. Therefore,
$\Pol\{\rho, \sigma\}\subsetneq \Pol\{\rho,
\beta_{m_{0}}\}\subsetneq \Pol\rho$, contradicting the maximality of
$\Pol\{\rho, \sigma\}$ in $\Pol\rho$.

b) Now we suppose that $\beta_{k}=E_{k}^{k}$. Let $j\in E_{k}$;
since $(j,j+1,\ldots,k-1,0,1,\ldots,j-1)\in E_{k}^{k}=\beta_{k}$,
there exists $u_{j}\in E_{k}$ such that for all
$\{i_{1},\ldots,i_{h-1}\}\subseteq E_{k}$,
$(i_{1},\ldots,i_{h-2},j,u_{j})\in\gamma$ and
$(i_{1},\ldots,i_{h-1},u_{j})\in\rho$. Hence $u_{j}\in C_{\rho}$.
Recall that $(c_{1},\ldots,c_{h})=(c,a_{2},\ldots,
a_{h})\in\rho\setminus\sigma$. Let us show the contradiction
$(c_{1},\ldots,c_{h})=(c,a_{2},\ldots, a_{h})\in\gamma$. Let
$u_{a_{h}}$ be a central element of $\rho$ related to $a_{h}$ as
above. For $i\in\{1,\ldots,h-1\}$, we have
$(c_{1},\ldots,c_{i-1},u_{a_{h}},c_{i+1},\ldots,c_{h}=a_{h})\in~\gamma$
because
$(c_{1},c_{2},\ldots,c_{i-1},c_{i+1},\ldots,c_{h}=a_{h},u_{a_{h}})\in\gamma$
and $\gamma$ is totally symmetric. It remains to show that
$(c,c_{2},c_{3},\ldots,c_{h-1},u_{a_{h}})\in\gamma$ to conclude that
$(c_{1},\ldots,c_{h})\in\gamma_{1}$. Since $c, u_{a_{h}}\in
C_{\rho}$, we have $(\omega,c_{2},c_{3},\ldots,c_{h-1},u_{a_{h}})$,
$(c,\omega,c_{3},\ldots,c_{h-1},u_{a_{h}})$,$\ldots$,
$(c,c_{2},c_{3},\ldots,c_{h-1},\omega)\in\gamma$; thus,
$(c,c_{2},c_{3},\ldots,c_{h-1},u_{a_{h}})\in\gamma_{1}$.
Consequently,
$(c,c_{2},$\\
$\ldots,c_{h-1},u_{a_{h}})\in(\gamma_{1}\cap\rho)=~\gamma$. From
$(c_{1},\ldots,c_{i-1},u_{a_{h}},c_{i+1}\ldots,c_{h})\in\gamma$ for
$1\leq i\leq h$, we conclude that $(c_{1},c_{2},c_{3},\ldots,c_{h})
\in(\gamma_{1}\cap\rho)=\gamma$, contradicting the choice of the
tuple $(c_{1},c_{2},c_{3},\ldots,c_{h})$.
\end{proof}
Hence, we are left with subcase (ii)

\begin{lemma}\label{lem 16}
Under the assumptions of Proposition \ref{prop 3} and
$\rho\cap\sigma\not\in\{\rho,\sigma\}$, the subcase
$\gamma\subsetneq(\gamma_{1}\cap\rho)$
 is impossible.
\end{lemma}

\begin{proof}
Since $\gamma\varsubsetneq\beta=(\gamma_{1}\cap\rho)$, choose
$(x_{1},\ldots,x_{h})\in\beta=(\gamma_{1}\cap\rho)\setminus\gamma$;
there exists $u\in E_{k}$ such that for all $1\leq i\leq h$,
$(x_{1},\ldots,x_{i-1},u,x_{i+1},x_{h})\in\gamma$. Let
$(u_{1},\ldots,u_{h})\in\sigma\setminus\gamma$(due to
$\gamma\varsubsetneq\sigma$). We will show that
$\Pol\{\rho,\sigma\}\subsetneq \Pol\{\rho, \gamma\}\subsetneq
\Pol\rho$. The unary operation $f$ defined by
\[
f(x)=
\begin{cases}
x_{i} & \text{if }x=u_{i}\text{ for some }1\leq i\leq h,\\
u & \text{otherwise;}
\end{cases}
\]
belongs to $\Pol\{\rho,\gamma\}\setminus\Pol\sigma$ because
$(x_{1},\ldots,x_{h})\in\rho\setminus\gamma$,
$\{x_{1},\ldots,x_{h}\}^{h-1}\times\{u\}\subseteq\gamma$,
$\img(f)=\{x_{1},\ldots,x_{h},u\}$, and
$(u_{1},\ldots,u_{h})\in\sigma\setminus\gamma$, but
$f(u_{1},\ldots,u_{h})=(x_{1},\ldots,x_{h})\notin\sigma$ (due to
$(x_{1},\ldots,x_{h})\in\rho\setminus\gamma)$ ; hence,
$\Pol\{\rho,\sigma\}\varsubsetneq\Pol\{\rho,\gamma\}$. Moreover,
consider the operation $g$ defined by
\[ g(x)=
\begin{cases}
x_{i} & \text{if }x=x_{i}\text{ for some }2\leq i\leq h,\\
x_{1} & \text{otherwise.}
\end{cases}
\]
Since $(x_{1},\ldots,x_{h})\in\rho$,
$(u,x_{2},\ldots,x_{h})\in\gamma$ and
$g(u,x_{2},\ldots,x_{h})=(x_{1},\ldots,x_{h})\notin\gamma$, we have
$g\in\Pol\rho\setminus\Pol\gamma$. Therefore,
$\Pol\{\rho,\sigma\}\varsubsetneq\Pol\{\rho,\gamma\}\varsubsetneq\Pol\rho$,
contradicting the maximality of $\Pol\{\rho, \sigma\}$ in
$\Pol\rho$.
\end{proof}

\begin{proof}[Proof (of Proposition \ref{prop 3})]
Combining Lemmas \ref{lem 14}--\ref{lem 16}  we obtain the result.
\end{proof}

 The above lemma closes the case $h=s$. Now we focus our attention on the case
 $h<s$. We begin this case by the following lemma:
\begin{lemma}\label{lem e}
Let $k\geq 3$, $\alpha$ an $n$-ary central relation and $\beta$ an
$m$-ary central relation on $E_{k}$ such that $2\leq n<m$. If
$\alpha_{n}:=\{(a_{1},\ldots,a_{n})\in E_{k}^{n}: \exists u\in
E_{k}, \forall 1\leq i_{1}<\cdots<i_{n-1}\leq n,
(a_{i_{1}},\ldots,a_{i_{n-1}},u)\in\alpha\text{ and
}(a_{1},\ldots,a_{n},\underset{(m-n)\text{
times}}{\underbrace{u,\ldots,u}})\in\beta\}$, then
$C_{\beta}\subseteq C_{\alpha}$.
\end{lemma}
\begin{proof}
Let $\omega\in C_{\beta}, a_{2},\ldots,a_{n}\in E_{k}$, and $c\in
C_{\alpha}$. We have $(\omega,a_{2},\ldots,a_{n})\in\alpha_{n}$
because $(\omega,a_{2},\ldots,a_{n},\underset{(m-n)\text{
times}}{\underbrace{c,\ldots,c}})\in\beta$, and  for all $2\leq
i_{1}<\cdots<i_{n-1}\leq n$,
$(\omega,a_{i_{1}},\ldots,a_{i_{n-2}},c)$,
$(a_{2},\ldots,a_{n},c)\in\alpha$. Therefore,
$(\omega,a_{2},\ldots,a_{n})\in\alpha_{n}=\alpha$
 and $\omega\in C_{\alpha}$. Thus, $C_{\beta}\subseteq C_{\alpha}$.
\end{proof}

\begin{proposition}\label{prop 4}
Let $k\geq 3$, $\rho$ an $h$-ary central relation and $\sigma$ an
$s$-ary central relation on $E_{k}$ such that $2\leq h<s$. If
$\Pol\{\rho, \sigma\}$ is a maximal subclone of $\Pol\rho$, then
$\sigma$ is of type V.
\end{proposition}
The proof of Proposition \ref{prop 4} follows from Lemmas \ref{lem
26}--\ref{lem 30}.
For $h\leq t\leq s-1$, we set\\
$\theta_{t}=\{(a_{1},\ldots,a_{t})\in E_{k}^{t}:\exists u\in E_{k}:
\forall  1\leq i_{1}<\cdots<i_{h-1}\leq t:
(a_{i_{1}},\ldots,a_{i_{h-1}},u)\in \rho \wedge
(a_{1},\ldots,a_{t},\underset{(s-t)\text{ times}}
{\underbrace{u,\ldots,u}})\in\sigma\}$. Clearly
$\rho\subseteq\theta_{h}\subseteq E_{k}^{h}$ and $\theta_{h}$ is
totally symmetric.
 We will show that $C_{\rho}\cap C_{\sigma}\neq\emptyset$. $\theta_{h}$ fulfills one of the following three cases:

 (4.1) $\rho=\theta_{h}$, (4.2) $\rho\subsetneq \theta_{h}\subsetneq E_{k}^{h}$ and (4.3) $\theta_{h}=E_{k}^{h}$.
\begin{lemma}\label{lem 26}
Under the assumptions of Proposition \ref{prop 4} and
$\theta_{h}=\rho$, we obtain $C_{\sigma}\subseteq C_{\rho}$.
\end{lemma}
\begin{proof}
It follows from Lemma \ref{lem e}.
\end{proof}
\begin{lemma}\label{lem 27}
Under the assumptions of Proposition \ref{prop 4}, the case
$\rho\varsubsetneq \theta_{h}\varsubsetneq E_{k}^{h}$ is impossible.
  \end{lemma}
 \begin{proof}
Since $\rho\varsubsetneq\theta_{h}\varsubsetneq E_{k}^{h}$,
$\theta_{h}$ is a central relation of type III and there is
$(u_{1},\ldots,u_{h})$ in $E_{k}^{h}\setminus\theta_{h}$. From
$\sigma\varsubsetneq E_{k}^{s}$, choose $(v_{1},\ldots, v_{s})\in
E_{k}^{s}\setminus\sigma$. Let $W'=\{(i_{1},\ldots,i_{h}): 1\leq
i_{1}<i_{2}<\cdots<i_{h}\leq s\}$. For the reason of simpler
notations we set $W'=\{(i_{1}^{j},\ldots,i_{h}^{j}):1\leq j\leq
q\}$. Set $\mbox{\boldmath{$y$}}_{j}=(x_{j,1},\ldots,x_{j,s})$ with
 \[
x_{j,p}=
\begin{cases}
u_{l} & \text{if }p=i_{l}^{j}\text{ for some }1\leq l\leq h, \\
u_{1} & \text{otherwise.}
\end{cases}
\]
Set $\mbox{\boldmath{$x$}}_{i}=(x_{1,i},\ldots,x_{q,i}),1\leq i\leq
s$. We have $(\mbox{\boldmath{$x$}}_{i_{1}},
\ldots,\mbox{\boldmath{$x$}}_{i_{h}})\not\in\theta_{h}$(moreover, it
is not in $\rho$) for $1\leq i_{1}<i_{2}<\ldots<i_{h}\leq s$
($\ast$) (due to $(u_{1},\ldots,u_{h})\not\in\theta_{h}$). Consider
the $q$-ary operation $f$ defined on $E_{k}$ by
\[
f(\mbox{\boldmath{$x$}})=
\begin{cases}
v_{i} & \text{if }\mbox{\boldmath{$x$}}=\mbox{\boldmath{$x$}}_{i}\text{ for some }1\leq i\leq s, \\
c & \text{otherwise.}
\end{cases}
\]
This operation is well defined, similarly as in Lemma \ref{lem 12},
and we have $f\in \Pol\{\rho,\theta_{h}\}$ (due to $c\in C_{\rho}$
and ($\ast$)). Furthermore,
$\{\mbox{\boldmath{$y$}}_{1},\ldots,\mbox{\boldmath{$y$}}_{q}\}\subseteq\sigma$(due
to $s>h$ and $\sigma$ totally reflexive), but
$(f(\mbox{\boldmath{$x$}}_{1}),\ldots,f(\mbox{\boldmath{$x$}}_{s}))=(v_{1},\ldots,v_{s})\not\in
\sigma$; so, $f\not\in \Pol\sigma$. Therefore $\Pol\{\rho,
\sigma\}\varsubsetneq \Pol\{\rho, \theta_{h}\}\varsubsetneq
\Pol\rho$, contradicting the maximality of $\Pol\{\rho, \sigma\}$ in
$\Pol\rho$.
\end{proof}

\begin{lemma}\label{lem 28}
Under the assumptions of Proposition \ref{prop 4} and
$\theta_{h}=E_{k}^{h}$, we have $\theta_{s-1}=E_{k}^{s-1}$.
\end{lemma}
\begin{proof}
Let $c\in C_{\rho}$ and $\omega\in C_{\sigma}$. Moreover $h<n\leq
s-1$ was chosen as the least index such that $\theta_{n}\neq
E_{k}^{n}$. So we can pick a tuple $(d_{1},\ldots,d_{n})\in
E_{k}^{n}\setminus\theta_{n}$. Because this tuple does not belong to
$\theta_{n}$ and $c\in C_{\rho}$, it follows that
$(d_{1},\ldots,d_{n},\underset{(s-n)\text{
times}}{\underbrace{c,\ldots,c}})\notin~\sigma$. This implies that
$n=s-1$, $\omega\notin\{d_{1},\ldots,d_{s-1},c\}$,
$c\notin\{d_{1},\ldots,d_{s-1}\}$, and that $d_{i}\neq d_{j}$ for
$1\leq i<j\leq s-1$. Moreover, we will show that $\theta_{s-1}$ is a
central relation. By definition, $\theta_{s-1}$ is totally reflexive
and totally symmetric. Let $v_{2},\ldots,v_{s-1}\in E_{k}$, and
$\{i_{1},\ldots,i_{h-1}\}\subseteq\{2,\ldots,s-1\}$; we have
$(v_{i_{1}},\ldots,v_{i_{h-1}},c)$,
$(v_{i_{1}},\ldots,v_{i_{h-2}},c,c)\in\rho$ because $c\in C_{\rho}$
and $\rho$ is totally reflexive. Moreover,
$(c,v_{2},\ldots,v_{s-1},c)\in\sigma$. Therefore,
$(c,v_{2},\ldots,v_{s-1})\in\theta_{s-1}$ and $\theta_{s-1}$ is a
central relation with $c\in C_{\theta_{s-1}}$.

Let us define a unary operation $f$ by $f(x)=x$ if
$x\in\{d_{1},\ldots,d_{s-1}\}$ and $f(x)=c$ if
$x\notin\{d_{1},\ldots,d_{s-1}\}$. Since $\omega\in C_{\sigma}$, we
have $(d_{1},\ldots,d_{s-1},\omega)\in\sigma$, but $f$ maps this
tuple to $(d_{1},\ldots,d_{s-1},c)\notin\sigma$. Therefore,
$f\notin\Pol\sigma$. We shall demonstrate further that
$f\in\Pol\{\rho,\theta_{n}\}$.

First let $(x_{1},\ldots,x_{h})\in\rho$. If there is $1\leq i\leq h$
such that $x_{i}\notin\{d_{1},\ldots,d_{s-1}\}$, then
$(f(x_{1}),\ldots,f(x_{h}))\in\rho$ because $c\in C_{\rho}$.
Otherwise, $\{x_{1},\ldots,x_{h}\}\subseteq\{d_{1},\ldots,d_{n}\}$,
so $(f(x_{1}),\ldots,f(x_{h}))=(x_{1},\ldots,x_{h})\in\rho$.
Consequently, we have $f\in\Pol\rho$.

Now we consider $(x_{1},\ldots,x_{s-1})\in\theta_{s-1}$. If
$x_{i}=x_{j}$ for some $1\leq i<j\leq s-1$, then
$(f(x_{1}),\ldots,f(x_{s-1}))\in\theta_{s-1}$ because we already
know that $\theta_{s-1}$ is totally reflexive. Likewise, if there
are $1\leq i<j\leq s-1$ such that
$x_{i},x_{j}\notin\{d_{1},\ldots,d_{s-1}\}$, then we are again done
by total reflexivity of $\theta_{n}$. If there is no index $1\leq
i\leq s-1$ such that $x_{i}\notin\{d_{1},\ldots,d_{s-1}\}$, then
$\{x_{1},\ldots,x_{s-1}\}\subseteq\{d_{1},\ldots,d_{s-1})$, so, as
above,
$(f(x_{1}),\ldots,f(x_{s-1}))=(x_{1},\ldots,x_{s-1})\in\theta_{s-1}$.
The remaining case is when there is exactly one $1\leq i\leq s-1$
such that $x_{i}\notin\{d_{1},\ldots,d_{s-1}\}$ and
$\{x_{1},\ldots,x_{i-1},x_{i+1},\ldots,x_{s-1}\}\subseteq\{d_{1},\ldots,d_{s-1}\}$
with all $x_{j}$ ($j\neq i$) being pairwise distinct. Consequently,
$(f(x_{1}),\ldots,f(x_{s-1}))=(d_{l_{1}},\ldots,d_{l_{i-1}},c,d_{l_{i}},\ldots,d_{l_{s-2}})\in\theta_{s-1}$
for some $\{l_{1},\ldots,l_{s-1}\}\subseteq\{1,\ldots,s-1\}$ (due to
$c\in C_{\theta_{s-1}}$). Therefore, $f\in\Pol\theta_{s-1}$.

From $s-1>h$, $\rho$ and $\theta_{h}$ are two different central
relations. Hence, $\Pol\rho $ and $\Pol\theta_{s-1}$ are two
different maximal clones. Therefore, $\Pol\{\rho,
\theta_{s-1}\}\varsubsetneq \Pol\rho$, contradicting the maximality
of $\Pol\{\rho, \sigma\}$ in $\Pol\rho$.
\end{proof}

Now we continue our investigation with the fact that $\theta_{s-1}=E_{k}^{s-1}$. Therefore for
 every $(a_{2},\ldots,a_{s})\in E_{k}^{s-1}=\theta_{s-1}$, there exists $u\in E_{k}$ such that $(a_{2},\ldots,a_{s},u)\in\sigma$ and for all $2\leq i_{1}<\ldots<i_{h-1}\leq s$,
  $(a_{i_{1}},\ldots,a_{i_{h-1}},u)\in\rho$.
\begin{lemma}\label{lem 29}
Under the assumptions of Proposition \ref{prop 4} and $\theta_{s-1}=E_{k}^{s-1}$, we obtain $C_{\rho}\cap C_{\sigma}\neq \emptyset$.
\end{lemma}

\begin{proof}
Suppose that  $C_{\rho}\cap C_{\sigma}=\emptyset$. Let $c\in
C_{\rho}$, then there exist $a_{2},\ldots,a_{s}\in E_{k}$ such that
$(c,a_{2},\ldots,a_{s})\not\in\sigma$. Since
$\theta_{s-1}=E_{k}^{s-1}$, there exists $u\in E_{k}$ such that
$(a_{2},\ldots,a_{s},u)\in\sigma$ and for all $2\leq
i_{1}<\ldots<i_{h-1}\leq s,
(a_{i_{1}},\ldots,a_{i_{h-1}},u)\in\rho$.

 Suppose that there exist $2\leq i_{1}<\ldots<i_{s-2}\leq s$ such that $(a_{i_{1}},\ldots,a_{i_{h-2}},c,u)\not\in~\sigma$.
Without loss of generality, we suppose that
$(c,u,a_{3},\ldots,a_{s})\not\in\sigma$. We set

$\gamma_{s}=\{(x_{1},\ldots,x_{s})\in\sigma: \forall
i_{1},\ldots,i_{h-1}\in\underline{s},
(x_{1},x_{i_{1}},\ldots,x_{i_{h-1}})\in\rho$ and

$ (x_{2},x_{i_{1}},\ldots, x_{i_{h-1}}) \in\rho \}$.

 We will show
that $\Pol\{\rho,\sigma\}\subsetneq \Pol\gamma_{s}\subsetneq
\Pol\rho$.
 Since $\gamma_{s}\in [\{\rho,\sigma\}]$, we have $\Pol\{\rho,\sigma\}\subseteq \Pol\gamma_{s}$.
 Let $f\in \Pol\gamma_{s}$ be an $n$-ary operation, $\mbox{\boldmath{$b$}}_{i}=(b_{i,1},\ldots,b_{i,h})\in\rho,1\leq i\leq n$,
 then $\mbox{\boldmath{$b$}}'_{i}=(b_{i,1},\ldots,b_{i,h},\ldots,b_{i,h})\in~\gamma_{s}$
and

$(f(b_{1,1},\ldots,b_{n,1}),\ldots,f(b_{1,h},
\ldots,b_{n,h}),\ldots,f(b_{1,h},\ldots,b_{n,h}))\in~\gamma_{s}$; by
the definition of $\gamma_{s}$, we deduce that
$(f(b_{1,1},\ldots,b_{n,1}),\ldots,f(b_{1,h},\cdot,b_{n,h}))\in~\rho$.
 So, $f\in \Pol\rho$ and $\Pol\gamma_{s}\subseteq \Pol\rho$.
 Therefore, $\Pol\{\rho,\sigma\}\subseteq \Pol\gamma_{s}\subseteq \Pol\rho$.
Let $(b_{1},\ldots,b_{h})\in E_{k}^{h}\setminus\rho$,
$W=\{(i_{1},\ldots,i_{h}): 1\leq i_{1}<\ldots<i_{h}\leq
s\}=\{(i_{1}^{j},\ldots,i_{h}^{j}): 1\leq j\leq q\}$ . For
$j=1,\ldots,q$ we set
$\mbox{\boldmath{$y$}}_{j}=(x_{j,1},\ldots,x_{j,s}), 1\leq j\leq q$,
where:
\[
x_{j,p}=
\begin{cases}
b_{l} & \text{if }p=i_{l}^{j}\text{ for some } 1\leq l\leq h,\\
b_{1} & \text{otherwise.}
\end{cases}
\]
For all $1\leq i\leq s$ we set
$\mbox{\boldmath{$x$}}_{i}=(x_{1,i},\ldots,x_{q,i})$. Consider the
$q$-ary operation $f$ defined by
\[
f(\mbox{\boldmath{$x$}})=
\begin{cases}
c & \text{if }\mbox{\boldmath{$x$}}=\mbox{\boldmath{$x$}}_{2}, \\
a_{i} & \text{if }\mbox{\boldmath{$x$}}=\mbox{\boldmath{$x$}}_{i}~\text{ for some }3\leq i\leq s, \\
u & \text{otherwise.}
\end{cases}
\]
$\sigma$ is totally reflexive and $s>h$; so by the construction of
$\mbox{\boldmath{$y$}}_{j}$ , we have
$\{\mbox{\boldmath{$y$}}_{1},\ldots,\mbox{\boldmath{$y$}}_{q}\}\subseteq\sigma$
and
 $f(\mbox{\boldmath{$y$}}_{1},\ldots,\mbox{\boldmath{$y$}}_{q})=(f(\mbox{\boldmath{$x$}}_{1}),\ldots,f(\mbox{\boldmath{$x$}}_{s}))=(u,c,a_{3},\ldots,a_{s})\not\in\sigma$,
so $f\not\in \Pol\sigma$. Let
$(b_{j,1},\ldots,b_{j,s})\in\gamma_{s}, 1\leq j\leq q$ and
 $\mbox{\boldmath{$d$}}_{i}=(b_{1,i},\ldots,b_{q,i}),1\leq i\leq s$. We will show
 that
 $(f(\mbox{\boldmath{$d$}}_{1}),\ldots,f(\mbox{\boldmath{$d$}}_{s}))\in\gamma_{s}$.
 By construction of $\mbox{\boldmath{$d$}}_{i}$, for all $1\leq i_{1}<\ldots<i_{h-1}\leq s$, we
 have
 $(\mbox{\boldmath{$d$}}_{1},\mbox{\boldmath{$d$}}_{i_{1}},\ldots,\mbox{\boldmath{$d$}}_{i_{h-1}})$,
 $(\mbox{\boldmath{$d$}}_{2},\mbox{\boldmath{$d$}}_{i_{1}},\ldots,\mbox{\boldmath{$d$}}_{i_{h-1}})\in~\rho$.
Let $i\in\{1,2\},i_{1},\ldots,i_{h-1}\in\{1,\ldots,s\}$. Again from
the construction of $\mbox{\boldmath{$d$}}_{p}, 1\leq p\leq s$, we
have also
$(\mbox{\boldmath{$d$}}_{i},\mbox{\boldmath{$d$}}_{i_{1}},\ldots,\mbox{\boldmath{$d$}}_{i_{h-1}})\in~\rho$.
If
$\Card(\{\mbox{\boldmath{$d$}}_{i},\mbox{\boldmath{$d$}}_{i_{1}},\ldots,\mbox{\boldmath{$d$}}_{i_{h-1}}\})\leq
h-1$, then
$(f(\mbox{\boldmath{$d$}}_{i}),f(\mbox{\boldmath{$d$}}_{i_{1}}),\ldots,f(\mbox{\boldmath{$d$}}_{i_{h-1}}))\in\rho$
because $\rho$ is totally reflexive. Otherwise,
$\Card(\{\mbox{\boldmath{$d$}}_{i},\mbox{\boldmath{$d$}}_{i_{1}},\ldots,\mbox{\boldmath{$d$}}_{i_{h-1}}\})=
h$, then by the construction of $\mbox{\boldmath{$x$}}_{p},1\leq
p\leq s$, there exists $j\in\{1,\ldots,h\}$ such that
$f(\mbox{\boldmath{$d$}}_{i_{j}})=u$ and
$(f(\mbox{\boldmath{$d$}}_{i}),f(\mbox{\boldmath{$d$}}_{i_{1}}),\ldots,f(\mbox{\boldmath{$d$}}_{i_{h-1}}))\in\rho$
because
$\{f(\mbox{\boldmath{$d$}}_{i}),f(\mbox{\boldmath{$d$}}_{i_{1}}),\ldots,f(\mbox{\boldmath{$d$}}_{i_{j-1}}),f(\mbox{\boldmath{$d$}}_{i_{j+1}}),\ldots,
f(\mbox{\boldmath{$d$}}_{i_{h-1}})\}\subseteq\{c,a_{3},\ldots,a_{s}\}$.
It remains to show that
$(f(\mbox{\boldmath{$d$}}_{1}),\ldots,f(\mbox{\boldmath{$d$}}_{s}))\in\sigma$.
If
$\Card(\{\mbox{\boldmath{$d$}}_{1},\ldots,\mbox{\boldmath{$d$}}_{s}\})\leq
s-1$, then
$(f(\mbox{\boldmath{$d$}}_{1}),\ldots,f(\mbox{\boldmath{$d$}}_{s}))\in\sigma$,
because $\sigma$ is totally reflexive. Otherwise
$\Card(\{\mbox{\boldmath{$d$}}_{1},\ldots,\mbox{\boldmath{$d$}}_{s}
\})=s$. Suppose that
$\Card(\{f(\mbox{\boldmath{$d$}}_{1}),\ldots,f(\mbox{\boldmath{$d$}}_{s})\})=s$,
then
$\{f(\mbox{\boldmath{$d$}}_{1}),\ldots,f(\mbox{\boldmath{$d$}}_{s})\}=\{u,c,a_{3},$\\
$\ldots,a_{s}\}$
 and we have the following two cases:

(1)
$\{\mbox{\boldmath{$d$}}_{1},\ldots,\mbox{\boldmath{$d$}}_{s}\}=\{\mbox{\boldmath{$x$}}_{2},\ldots,\mbox{\boldmath{$x$}}_{s}\}\cup\{\mbox{\boldmath{$x$}}\},
\mbox{\boldmath{$x$}}\not\in
\{\mbox{\boldmath{$x$}}_{1},\ldots,\mbox{\boldmath{$x$}}_{s}\}$
 and (2) $\{\mbox{\boldmath{$d$}}_{1},\ldots,\mbox{\boldmath{$d$}}_{s}\}=\{\mbox{\boldmath{$x$}}_{1},\ldots,\mbox{\boldmath{$x$}}_{s}\}$.

If
$\{\mbox{\boldmath{$d$}}_{1},\ldots,\mbox{\boldmath{$d$}}_{s}\}=\{\mbox{\boldmath{$x$}}_{1},\ldots,\mbox{\boldmath{$x$}}_{s}\}$,
then
 $(\mbox{\boldmath{$x$}}_{i_{1}},\ldots,\mbox{\boldmath{$x$}}_{i_{h}})=(\mbox{\boldmath{$d$}}_{1},\ldots,\mbox{\boldmath{$d$}}_{h})\in\rho$,
 in contradiction with the construction of $\mbox{\boldmath{$x$}}_{i}$. So this case cannot occur.

If
$\{\mbox{\boldmath{$d$}}_{1},\ldots,\mbox{\boldmath{$d$}}_{s}\}=\{\mbox{\boldmath{$x$}}_{2},\ldots,\mbox{\boldmath{$x$}}_{s}\}\cup\{\mbox{\boldmath{$x$}}\},
\mbox{\boldmath{$x$}}\not\in\{\mbox{\boldmath{$x$}}_{1},
\ldots,\mbox{\boldmath{$x$}}_{s}\}$,
 then there exist $i_{1},\ldots,i_{h-1}\in\{2,\ldots,s\}$,
$\Card(\{ i_{1},\ldots,i_{h-1}\})=h-1$, such that
$(\mbox{\boldmath{$x$}}_{j_{1}},\ldots,\mbox{\boldmath{$x$}}_{j_{h}})=(\mbox{\boldmath{$d$}}_{1},\mbox{\boldmath{$d$}}_{i_{1}},\ldots,\mbox{\boldmath{$d$}}_{i_{h-1}})
$ in $\rho,1\leq j_{1}<\ldots<j_{h}\leq s$ or
$(\mbox{\boldmath{$x$}}_{j_{1}},\ldots,\mbox{\boldmath{$x$}}_{j_{h}})=(\mbox{\boldmath{$d$}}_{2},\mbox{\boldmath{$d$}}_{i_{1}},\ldots,
\mbox{\boldmath{$d$}}_{i_{h-1}}) \in\rho,1\leq
j_{1}<\ldots<j_{h}\leq s$,
 in contradiction with the construction of $\mbox{\boldmath{$x$}}_{i},1\leq i\leq s$.
So the case
$\Card(\{f(\mbox{\boldmath{$d$}}_{1}),\ldots,f(\mbox{\boldmath{$d$}}_{s})\})=s$
cannot occur. By the total reflexivity of $\sigma$ we deduce that
 $(f(\mbox{\boldmath{$d$}}_{1}),\ldots,f(\mbox{\boldmath{$d$}}_{s}))\in\sigma$ and $f\in \Pol\gamma_{s}$.
 Therefore, $\Pol\{\rho,\sigma\}\subsetneq \Pol\gamma_{s}$.
We will show that $\Pol\gamma_{s}\subsetneq \Pol\rho$. We consider
$\mbox{\boldmath{$y$}}_{1}=(u,u,a_{3},\ldots,a_{s}),\mbox{\boldmath{$y$}}_{2}=(u,c,a_{3},a_{3},a_{5},\ldots,a_{s})\in\gamma_{s}$
and the binary operation $f$ defined on $E_{k}$ by
\[
f(\mbox{\boldmath{$x$}})=
\begin{cases}
 c & \text{if }\mbox{\boldmath{$x$}}=(u,c), \\
 a_{4} & \text{if }\mbox{\boldmath{$x$}}=(a_{4},a_{3}),\\
 a_{i} & \text{if }\mbox{\boldmath{$x$}}=(a_{i},a_{i})\text{ for some }3\leq i\leq s, \\
 u & \text{otherwise.}
\end{cases}
\]
$\{\mbox{\boldmath{$y$}}_{1},\mbox{\boldmath{$y$}}_{2}\}\subseteq\gamma_{s}$
(due to the total reflexivity of $\sigma$, the properties satisfying
$u$ and $c\in C_{\rho}$) and
$f(\mbox{\boldmath{$y$}}_{1},\mbox{\boldmath{$y$}}_{2})=
 (f(u,u),f(u,c),f(a_{3},a_{3}),f(a_{4},a_{3}),f(a_{5},a_{5}),
\ldots,f(a_{s},a_{s}))$=
$(u,c,a_{3},\ldots,a_{s})\not\in\gamma_{s}$; so, $f\not\in
\Pol\gamma_{s}$. Let
$\mbox{\boldmath{$x$}}=(x_{1},\ldots,x_{h}),\mbox{\boldmath{$y$}}=(y_{1},\ldots,y_{h})\in\rho$.
We set $\mbox{\boldmath{$d$}}_{i}=(x_{i},y_{i}),1\leq i\leq h$. If
$\Card(\{\mbox{\boldmath{$d$}}_{1},\ldots,\mbox{\boldmath{$d$}}_{h}\})\leq
h-1$, then
 $(f(\mbox{\boldmath{$d$}}_{1}),\ldots,f(\mbox{\boldmath{$d$}}_{h}))\in\rho$ because $\rho$ is totally reflexive.
Otherwise
$\Card(\{\mbox{\boldmath{$d$}}_{1},\ldots,\mbox{\boldmath{$d$}}_{h}\})=h$.
 If there exists $i\in\{1,\ldots,h\}$ such that
 $\mbox{\boldmath{$d$}}_{i}=(x_{i},y_{i})\in\{(u,u),(u,c)\}$ or
 $\mbox{\boldmath{$d$}}_{i}\notin\{(u,u),(u,c),(a_{3},a_{3}),$\\
 $(a_{4},a_{3}),(a_{5},a_{5}),\ldots,(a_{s},a_{s})\}$
 with all $\mbox{\boldmath{$d$}}_{i}$ being pairwise distinct, then $f(\mbox{\boldmath{$d$}}_{i})\in\{u,c\}$ and
 $(f(\mbox{\boldmath{$d$}}_{1}),\ldots,f(\mbox{\boldmath{$d$}}_{h}))\in\rho$. The remaining case is when
 $\mbox{\boldmath{$d$}}_{i}\in\{(a_{3},a_{3}), (a_{4},a_{3}),$\\
 $(a_{5},a_{5}),\ldots, (a_{s},a_{s})\}$ for $1\leq i\leq h$. In
 this case,
 $f(\mbox{\boldmath{$d$}}_{i})=\pi_{1}^{(2)}(\mbox{\boldmath{$d$}}_{i})=x_{i}$ for $1\leq i\leq h$.
 Hence,
 $(f(\mbox{\boldmath{$d$}}_{1}),\ldots,f(\mbox{\boldmath{$d$}}_{h}))=(\pi_{1}^{(2)}(\mbox{\boldmath{$d$}}_{1}),\ldots,\pi_{1}^{(2)}(\mbox{\boldmath{$d$}}_{h}))=
 (x_{1},\ldots,x_{h})\in~\rho$.
 So, $f\in \Pol\rho$ and $\Pol\gamma_{s}\subsetneq \Pol\rho$.
 Therefore,
$\Pol\{\rho,\sigma\}\subsetneq \Pol\gamma_{s}\subsetneq \Pol\rho$,
and we obtain a contradiction with the assumptions of Proposition~
\ref{prop 4}.

Suppose that for all $2\leq i_{1}<\ldots<i_{s-2}\leq s,
(c,u,a_{i_{1}},\ldots,a_{i_{s-2}})\in\sigma$.
 For $t\geq s$ we set

$\gamma'_{t}=\{(x_{1},\ldots,x_{t})\in E_{k}^{t}:\exists v\in E_{k}:
\forall i_{1},\cdots,i_{s-1}\in \underline{s}
 \text{ }(x_{i_{1}},\ldots,x_{i_{h-1}},v)\in\rho,$\\
$(x_{i_{1}},\ldots,x_{i_{s-1}},v) \in~\sigma\}$.

 Since $C_{\rho}\cap C_{\sigma}=\emptyset$, we have $\gamma'_{k}\neq E_{k}^{k}$.
Let $n_{0}=\min\{j\geq s:\gamma'_{j}\neq E_{k}^{j}\}$. Therefore,
$n_{0}\geq s$, $\gamma'_{n_{0}}$ is totally
 reflexive and totally symmetric. We will show that
$\Pol\{\rho,\sigma\}\subsetneq
\Pol\{\rho,\gamma'_{n_{0}}\}\subsetneq \Pol\rho$. Since
$\gamma'_{n_{0}}\in[\{\rho,\sigma\}]$, we have
$\Pol\{\rho,\sigma\}\subseteq \Pol\{\rho,\gamma'_{n_{0}}\}\subseteq
\Pol\rho$.

Recall that $(c,a_{2},\ldots,a_{s})\notin\sigma$ and there is $u\in
E_{k}$ such that for all $2\leq i_{1}<\cdots< i_{s-2}\leq s$,
$(c,a_{i_{1}},\ldots,a_{i_{s-2}},u)$,
$(a_{2},\ldots,a_{s},u)\in\sigma$ and
$(a_{i_{1}},\ldots,a_{i_{h-1}},u)\in\rho$. Hence,
$(c,a_{2},\ldots,a_{s})\in\gamma'_{s}\setminus\sigma$. Let
$\omega\in C_{\sigma}$; we have
$\omega\notin\{c,a_{2},\ldots,a_{s}\}$,
$(\omega,a_{2},\ldots,a_{s})\in\sigma$ and $a_{2},\ldots,a_{s}$ are
pairwise distinct. Let us defined a unary operation $f$ by $f(x)=x$
if $x\in\{a_{2},\ldots,a_{s}\}$ and $f(x)=c$ if
$x\notin\{a_{2},\ldots,a_{s}\}$. We have
$(\omega,a_{2},\ldots,a_{s})\in\sigma$, but
$f(\omega,a_{2},\ldots,a_{s})=(c,a_{2},\ldots,a_{s})\notin\sigma$.
Therefore, $f\notin\Pol\sigma$.

If $n_{0}>s$, then $f\in\Pol\gamma'_{n_{0}}$ because
$\img(f)=\{c,a_{2},\ldots,a_{s}\}$ and $\gamma'_{n_{0}}$ is totally
reflexive. If $n_{0}=s$, then $f\in\Pol\gamma'_{s}$ because
$\img(f)=\{c,a_{2},\ldots,a_{s}\}$,
$(c,a_{2},\ldots,a_{s})\in\gamma'_{s}$ and $\gamma'_{s}$ is totally
reflexive and totally symmetric. Moreover, $f\in\Pol\rho$ and
$\Pol\{\rho,\sigma\}\varsubsetneq \Pol\{\rho,\gamma'_{s}\}$.

Let $(b_{1},\ldots,b_{h})\in E_{k}^{h}\setminus\rho$ and consider
$W_{h}^{s}=\{(i_{1},\ldots,i_{h}): 1\leq
i_{1}<i_{2}<\ldots<i_{h}\leq s\}$. For the reason of simpler
notations we set $W_{h}^{s}=\{(i_{1}^{j},
  \ldots,i_{h}^{j}):1\leq j\leq q\}$. We set $\mbox{\boldmath{$y$}}_{j}=(b_{j,1},\ldots,b_{j,n}),1\leq
j\leq q$, with
\[
b_{j,p}=
\begin{cases}
b_{l} & \text{if }p=i_{l}^{j}\text{ for some } 1\leq l\leq h, \\
b_{1} & \text{otherwise.}
\end{cases}
\]
From $\gamma'_{n_{0}}\subsetneq E_{k}^{n_{0}}$,  let
$(v_{1},\ldots,v_{n_{0}})\in E_{k}^{n_{0}}\setminus\gamma_{n_{0}}$.
 Let $\mbox{\boldmath{$x$}}_{i}=(b_{1,i},\ldots,b_{q,i}),1\leq i\leq n_{0}$, and $f$ be  the $q$-ary operation defined on $E_{k}$ by
\[
f(\mbox{\boldmath{$x$}})=
\begin{cases}
v_{i} & \text{if }\mbox{\boldmath{$x$}}=\mbox{\boldmath{$x$}}_{i}\text{ for some } 1\leq i\leq n_{0}, \\
c & \text{otherwise.}
\end{cases}
\]
$\{\mbox{\boldmath{$y$}}_{1},\ldots,\mbox{\boldmath{$y$}}_{q}\}\subseteq\gamma'_{n_{0}}$
and
$(f(\mbox{\boldmath{$x$}}_{1}),\ldots,f(\mbox{\boldmath{$x$}}_{n}))=(v_{1},\ldots,v_{n_{0}})
\not\in\gamma'_{n_{0}}$. So, $f\not\in \Pol\gamma'_{n_{0}}$. Using
the construction of $\mbox{\boldmath{$x$}}_{i}$ we can show that
$f\in \Pol\rho$.
 So, $\Pol\{\rho,\gamma'_{n_{0}}\}\subsetneq \Pol\rho$.
 Thus, $\Pol\{\rho,\sigma\}\subsetneq \Pol\{\rho,\gamma'_{n_{0}}\}\subsetneq \Pol\rho$.
\end{proof}

We have shown that under the assumptions of Proposition \ref{prop 4}
together with $\theta_{s-1}= E_{k}^{s-1}$, we have  $C_{\rho}\cap
C_{\sigma}\neq \emptyset$. Let
$\gamma'=\{(b_{1},\ldots,b_{s})\in\sigma:
(b_{1},\ldots,b_{h})\in\rho\}$ and recall that
$\lambda=\{(b_{1},\ldots,b_{s})\in E_{k}^{s}:
(b_{1},\ldots,b_{h})\in\rho\}$ from Theorem \ref{theo 1}.

\begin{lemma}\label{lem 30}
Under the assumptions of Proposition \ref{prop 4} and $C_{\rho}\cap
C_{\sigma}\neq\emptyset$, we obtain a relation of type V.
\end{lemma}
\begin{proof}
Since $ \gamma'=\lambda\cap\sigma$, we have
$\gamma'\varsubsetneq\lambda$ or $\gamma'=\lambda$. Suppose that
$\gamma'\varsubsetneq\lambda$. Thus there is
$(v_{1},\ldots,v_{s})\in\lambda\setminus\gamma'$ i.e,
$(v_{1},\ldots,v_{h})\in\rho$ and
$(v_{1},\ldots,v_{s})\notin\sigma$. We will show that
$\Pol\{\rho,\sigma\}\varsubsetneq\Pol\gamma'\varsubsetneq\Pol\rho$.
Since $\gamma'\in[\{\rho,\sigma\}]$, we have
$\Pol\{\rho,\sigma\}\subseteq\Pol\gamma'\subseteq\Pol\rho$. Let us
show that $\Pol\{\rho,\sigma\}\varsubsetneq\Pol\gamma'$. We consider
$W_{h}^{s}=\{(i_{1},\ldots,i_{h}): 1\leq i_{1}<\ldots<i_{h}\leq s\}$
denoted for the reason of simpler notation by
$W_{h}^{s}=\{(i_{1}^{j},\ldots,i_{h}^{j}): 1\leq j\leq q\}$. Let us
choose $(b_{1},\ldots,b_{h})\in E_{k}^{h}\setminus\rho$. Since
$\rho$ is totally reflexive, we have $b_{l}\neq b_{l'}$ for $1\leq
l<l'\leq h$. Moreover, define
$\mbox{\boldmath{$y$}}_{j}=(b_{j,1},\ldots,b_{j,s}),1\leq j\leq q$,
with
\[
b_{j,p}=
\begin{cases}
b_{l} & \text{if }p=i_{l}^{j}\text{ for some }1\leq l\leq h,\\
b_{1} & \text{otherwise.}
\end{cases}
\]

We set $\mbox{\boldmath{$x$}}_{i}=(b_{1,i},\ldots,b_{q,i}),1\leq
i\leq s$, and note that
$\mbox{\boldmath{$x$}}_{i}\neq\mbox{\boldmath{$x$}}_{i'}$ for $1\leq
i<i'\leq s$ because $b_{1},\ldots,b_{h}$ are pairwise distinct.
Furthermore, we choose $c\in C_{\rho}\cap C_{\sigma}$. Consider the
$q$-ary operation $f$ defined on $E_{k}$ by
\[
f(\mbox{\boldmath{$x$}})=
\begin{cases}
v_{i} & \text{if }\mbox{\boldmath{$x$}}=\mbox{\boldmath{$x$}}_{i}\text{ for some }1\leq i\leq s,\\
c & \text{otherwise.}
\end{cases}
\]
The total reflexivity of $\sigma$ and $s>h$ imply that
$\{\mbox{\boldmath{$y$}}_{1},\ldots,\mbox{\boldmath{$y$}}_{q}\}\subseteq
\sigma$, furthermore
  $(f(\mbox{\boldmath{$x$}}_{1}),\ldots,f(\mbox{\boldmath{$x$}}_{s}))=(v_{1},\ldots,v_{s})\not\in\sigma$; so, $f\not\in \Pol\sigma$.
 Let $(a_{j,1},\ldots,a_{j,s})\in\gamma',1\leq j\leq q$.
We set $\mbox{\boldmath{$d$}}_{i}=(a_{1,i},\ldots,a_{q,i}),1\leq
i\leq s$.
 We will show that $(f(\mbox{\boldmath{$d$}}_{1}),\ldots,f(\mbox{\boldmath{$d$}}_{s}))\in\gamma'$.
If
$\Card(\{\mbox{\boldmath{$d$}}_{1},\ldots,\mbox{\boldmath{$d$}}_{h}\})\leq
h-1$, then the total reflexivity of $\rho$ and $\sigma$ implies that
 $(f(\mbox{\boldmath{$d$}}_{1}),\ldots,f(\mbox{\boldmath{$d$}}_{s}))\in\gamma'$. Otherwise,
 by the construction of $\mbox{\boldmath{$x$}}_{i},1\leq i\leq s$, there exists
 $i\in \underline{h}$ such that $f(\mbox{\boldmath{$d$}}_{i})=c$ and
 $(f(\mbox{\boldmath{$d$}}_{1}),\ldots,f(\mbox{\boldmath{$d$}}_{s}))\in\gamma'$.
So $f\in \Pol\gamma'$ and $\Pol\{\rho,\sigma\}\subsetneq
\Pol\gamma'$. Using $(b_{1},\ldots,b_{h})\in E_{k}^{h}\setminus\rho$
from above, we set $W=\{(i_{1},\ldots,i_{l}, j_{1},\ldots,j_{m}):
1\leq i_{1}<\ldots<i_{l}\leq h<h+1\leq j_{1}<\ldots<j_{m}\leq
s,l+m=h,1\leq l<h\}
=\{(i_{1}^{p},\ldots,i_{l_{p}}^{p},j_{1}^{p},\ldots,j_{m_{p}}^{p}):1\leq
p\leq q\}$. Let $c_{2},\ldots,c_{s}\in E_{k}$ such that
$\Card(\{c,c_{2},\ldots,c_{s}\})=s,
~\mbox{\boldmath{$y$}}_{p}=(b_{p,1},\ldots,b_{p,s}),1\leq p\leq q,
\mbox{\boldmath{$y$}}_{q+1}=(b_{q+1,1},\ldots,b_{q+1,s}):=(c,c_{2},\ldots,c_{s})$
and $\mbox{\boldmath{$x$}}_{i}=(b_{1,i},\ldots,b_{q+1,i}),1\leq
i\leq s$, where
\[
b_{p,l}=
\begin{cases}
b_{n} & \text{if }l=i_{n}^{p}\text{ for some }1\leq n\leq l_{p},\\
b_{l_{p}+n'} & \text{if }l=j_{n'}^{p}\text{ for some }1\leq n'\leq m_{p}, \\
b_{1} & \text{otherwise.}
\end{cases}
\]

$\mbox{\boldmath{$x$}}_{1},\ldots,\mbox{\boldmath{$x$}}_{s}$ are
pairwise distinct because of $\mbox{\boldmath{$y$}}_{q+1}$, thus
define the $(q+1)$-ary operation $f$ by
\[
f(\mbox{\boldmath{$x$}})=
\begin{cases}
v_{i} & \text{if }\mbox{\boldmath{$x$}}=\mbox{\boldmath{$x$}}_{i}\text{ for some }1\leq i\leq s, \\
c & \text{otherwise.}
\end{cases}
\]
 $\{\mbox{\boldmath{$y$}}_{1},\ldots,\mbox{\boldmath{$y$}}_{q},\mbox{\boldmath{$y$}}_{q+1}\}\subseteq \gamma'$ because for all $1\leq p\leq q+1$, there is $1\leq i<j\leq h
 $ such that $b_{p,i}=b_{p,j}=b_{1}$. Furthermore,
$f(\mbox{\boldmath{$y$}}_{1},\ldots,\mbox{\boldmath{$y$}}_{q+1})=(f(\mbox{\boldmath{$x$}}_{1}),\ldots,f(\mbox{\boldmath{$x$}}_{s}))=(v_{1},\ldots,v_{s})\not\in\gamma'$
 and $f\not\in \Pol\gamma'$.

  Let $(a_{j,1},\ldots,a_{j,h})\in\rho, 1\leq j\leq q+1$, and $\mbox{\boldmath{$d$}}_{i}=(a_{1,i},\ldots,a_{q+1,i})$ for $1\leq i\leq h$. We will show that
$(f(\mbox{\boldmath{$d$}}_{1}),\ldots,f(\mbox{\boldmath{$d$}}_{h}))\in\rho$.
If
$\Card\{\mbox{\boldmath{$d$}}_{1},\ldots,\mbox{\boldmath{$d$}}_{h}\})\leq
h-1$, then
$(f(\mbox{\boldmath{$d$}}_{1}),\ldots,f(\mbox{\boldmath{$d$}}_{h}))\in\rho$.
Otherwise
$\Card(\{\mbox{\boldmath{$d$}}_{1},\ldots,\mbox{\boldmath{$d$}}_{h}\})=h$
and we set
$L=\{\mbox{\boldmath{$d$}}_{1},\ldots,\mbox{\boldmath{$d$}}_{h}\}\cap\{\mbox{\boldmath{$x$}}_{1},\ldots,\mbox{\boldmath{$x$}}_{s}\}$.
 If $\Card(L) \leq h-1$, then there exists $i\in\underline{h}$ such
 that $f(\mbox{\boldmath{$d$}}_{i})=c$ and $(f(\mbox{\boldmath{$d$}}_{1}),\ldots,f(\mbox{\boldmath{$d$}}_{h}))\in\rho$.
Otherwise
$\{\mbox{\boldmath{$d$}}_{1},\ldots,\mbox{\boldmath{$d$}}_{h}\}=\{\mbox{\boldmath{$x$}}_{i_{1}},\ldots,\mbox{\boldmath{$x$}}_{i_{h}}\}$
for some $1\leq i_{1}<\ldots<i_{h}\leq s$.

If $\{i_{1},\ldots,i_{h}\}=\{1,\ldots,h\}$, then
$(f(\mbox{\boldmath{$d$}}_{1}),\ldots,f(\mbox{\boldmath{$d$}}_{h}))\in\rho$
because $(v_{1},\ldots,v_{h})\in\rho$.\\
If $\{i_{1},\ldots,i_{h}\}\subseteq\{h+1,\ldots,s\}$, then by the
construction
 of $\mbox{\boldmath{$x$}}_{i}$, for $l_{p}=1$ and $m_{p}=h-1$, there exist $j\in\{1,\ldots,q\}$
and $i_{0}\leq h$ such that
 $\{b_{j,i_{0}},b_{j,i_{2}},b_{j,i_{3}}\ldots,b_{j,i_{h}}\}=\{b_{1},\ldots,b_{h}\}$,
 $b_{j,i_{1}}=b_{1}$
and
$(\mbox{\boldmath{$d$}}_{1},\ldots,\mbox{\boldmath{$d$}}_{h})\in\rho$,
contradiction. Otherwise
$\{i_{1},\ldots,i_{h}\}=\{i_{1}^{p},\ldots,i_{l_{p}}^{p},j_{1}^{p},
\ldots,j_{m_{p}}^{p}\}, 1\leq i_{1}^{p}< \ldots <i_{l_{p}}^{p} \leq
h<h+1\leq j_{1}^{p} < \ldots<j_{m_{p}}^{p}\leq s$,
 then by permutation we have $(b_{1},\ldots,b_{h})\in\rho$, contradiction.
 So,
$f\in \Pol\rho$ and $\Pol\gamma'\subsetneq \Pol\rho$. Thus
$\Pol\{\rho,\sigma\}\subsetneq \Pol\gamma'\subsetneq \Pol\rho$,
contradicting the assumptions of Proposition \ref{prop 4}. Hence,
$\gamma'=\lambda$ and $\lambda\subsetneq\sigma$. Therefore we have a
relation of type~ V.
\end{proof}

We finish the completeness criterion  with the case $2\leq s<h$.
\begin{proposition}\label{prop 5}
Let $k\geq 3$, $\rho$ an $h$-ary central relation and $\sigma$ an
$s$-ary central relation on $E_{k}$ such that
 $2\leq s<h$.
If $\Pol\{\rho, \sigma\}$ is a maximal subclone of $\Pol\rho$, then
$\sigma$ is of type IV.
\end{proposition}
The proof of Proposition \ref{prop 5} is discussed in Lemmas
\ref{lem 31}--\ref{lem 34}.
 For $s\leq t\leq h-1$ we set

$\theta_{t}=\{(a_{1},\ldots,a_{t})\in E_{k}^{t}:\exists u\in E_{k}:
\forall i_{1},\ldots,i_{s-1}\in \underline{t}:
(a_{i_{1}},\ldots,a_{i_{s-1}},u)\in\sigma\wedge
(a_{1},\ldots,a_{t},\underset{(h-t)\text{
times}}{\underbrace{u,\ldots,u}})\in\rho\}$.

 Clearly
$\sigma\subseteq\theta_{s}$, $\theta_{s}$ is totally reflexive and
totally symmetric, and we have the following three subcases:

(1) $\sigma=\theta_{s}$, (2) $\sigma\subsetneq\theta_{s}\subsetneq
E_{k}^{s}$ and (3) $\theta_{s}=E_{k}^{s}$.
\begin{lemma}\label{lem 31}
Under the assumptions of Proposition \ref{prop 5} and $\theta_{s}=\sigma$, we have a relation of type IV.
\end{lemma}

\begin{proof}
It follows from Lemma \ref{lem e}.
\end{proof}

\begin{lemma}\label{lem 32}
Under the assumptions of Proposition \ref{prop 5}, the case
 $\sigma\subsetneq \theta_{s}\subsetneq E_{k}^{s}$ is impossible.
\end{lemma}
\begin{proof}
Let $\omega\in C_{\sigma}$. Since
$\sigma\varsubsetneq\theta_{s}\varsubsetneq E_{k}^{s}$, we have
$\theta_{s}$ is a central relation with $\omega\in C_{\theta_{s}}$
and there is $(u_{1},\ldots,u_{s})\in E_{k}^{s}\setminus\theta_{s}$.
There exists also $(v_{1},\ldots,
v_{s})\in\theta_{s}\setminus\sigma$; $v_{1},\ldots, v_{s}$ are
pairwise distinct and $\omega\notin\{v_{1},\ldots, v_{s}\}$ because
$\sigma$ is totally reflexive and $\omega\in C_{\sigma}$. Let us
define a unary operation on $E_{k}$ by $f(x)=x$ if
$x\in\{v_{2},\ldots,v_{s}\}$ and $f(x)=v_{1}$ if
$x\notin\{v_{2},\ldots,v_{s}\}$ . We have
$(\omega,v_{2},\ldots,v_{s})\in\sigma$ and
$(f(\omega),f(v_{2}),\ldots,f(v_{s}))=(v_{1},\ldots,
v_{s})\notin\sigma$. Therefore, $f\notin\Pol\sigma$. However $f$ is
in $\Pol\theta_{s}$ because $(v_{1},\ldots,v_{s})\in\theta_{s}$,
$\theta_{s}$ is totally reflexive, and
$\img(f)=\{v_{1},\ldots,v_{s}\}$. Since $s<h$ and $\rho$ totally
reflexive, we have $f\in\Pol\rho$. Hence,
$\Pol\{\rho,\sigma\}\varsubsetneq\Pol\{\rho,\theta_{s}\}$. It
remains to show that $\Pol\{\rho,\theta_{s}\}\varsubsetneq\Pol\rho$.
Moreover, $\rho$ and $\theta_{s}$ are two different central
relations. Hence, $\Pol\rho$ and $\Pol\theta_{s}$ are two different
maximal clones. Therefore,
$\Pol\{\rho,\theta_{s}\}\varsubsetneq\Pol\rho$, contradicting the
maximality of $\Pol\{\rho,\sigma\}$ in $\Pol\rho$.
\end{proof}

\begin{lemma}\label{lem 33}
Under the assumptions of Proposition \ref{prop 5} and $\theta_{s}=E_{k}^{s}$, we have $\theta_{h-1}=E_{k}^{h-1}$.
\end{lemma}
\begin{proof}
Let $c\in C_{\rho}$ and $\omega\in C_{\sigma}$. Moreover $s<n\leq
h-1$ was chosen as the least index such that $\theta_{n}\neq
E_{k}^{n}$. So we can pick a tuple $(d_{1},\ldots,d_{n})\in
E_{k}^{n}\setminus\theta_{n}$. Because this tuple does not belong to
$\theta_{n}$ and $\omega\in C_{\sigma}$, it follows that
$(d_{1},\ldots,d_{n},\underset{(h-n)\text{
times}}{\underbrace{\omega,\ldots,\omega}})\notin\rho$. This implies
that $n=h-1$, $c\notin\{d_{1},\ldots,d_{h-1},\omega\}$,
$\omega\notin\{d_{1},\ldots,d_{h-1}\}$, and that $d_{i}\neq d_{j}$
for $1\leq i<j\leq h-1$. Moreover,
$(d_{1},\ldots,d_{h-1},c)\in\rho$; thus there exist $1\leq
i_{1}<i_{2}<\cdots<i_{s-1}\leq h-1$, such that
$(d_{i_{1}},\ldots,d_{i_{s-1}},c)\notin\sigma$. Let us define a
unary operation $f$ by $f(x)=x$ if
$x\in\{d_{i_{1}},\ldots,d_{i_{s-1}}\}$ and $f(x)=c$ if
$x\notin\{d_{i_{1}},\ldots,d_{i_{s-1}}\}$. Since $\omega\in
C_{\sigma}$, we have
$(d_{i_{1}},\ldots,d_{i_{s-1}},\omega)\in\sigma$, but $f$ maps this
tuple to $(d_{i_{1}},\ldots,d_{i_{s-1}},c)\notin\sigma$. Therefore,
$f\notin\Pol\sigma$. We shall demonstrate further that
$f\in\Pol\{\rho,\theta_{h-1}\}$.

Let $(x_{1},\ldots,x_{h})\in\rho$, then
$(f(x_{1}),\ldots,f(x_{h}))\in\rho$ because $h>s$,
$\{f(x_{1}),\ldots,f(x_{h})\}$ is a subset of
$\{d_{i_{1}},\ldots,d_{i_{s-1}},c\}$, and $\rho$ is totally
reflexive. A similar argument shows that $f\in\Pol\theta_{h-1}$ (due
to $h-1>s$).

To complete our proof we shall show that
$\Pol\{\rho,\theta_{h-1}\}\varsubsetneq\Pol\rho$. We have
$(d_{1},\ldots,d_{h-1})\notin\theta_{h-1}$ and $\omega\in
C_{\sigma}$. Taking $u=\omega$, we observe that
$(c,d_{1},\ldots,d_{h-2})\in \theta_{h-1}$. Let us define a unary
function $f$ by $f(x)=x$ if $x=d_{i}$, for some $1\leq i\leq h-1$,
and $f(x)=d_{h-1}$ otherwise. $f$ does not preserve $\theta_{h-1}$
because $(c,d_{1},\ldots,d_{h-2})\in \theta_{h-1}$, but
$(f(c),f(d_{1}),\ldots,f(d_{h-2}))=(d_{h-1},d_{1},\ldots,d_{h-2})\notin
\theta_{h-1}$(due to $\theta_{h-1}$ totally symmetric). Since $\rho$
is totally reflexive and $\img(f)=\{d_{1},\ldots,d_{h-1}\}$, we have
$f\in\Pol\rho$, contradicting the maximality of
$\Pol\{\rho,\sigma\}$ in $\Pol\rho$.
\end{proof}
\begin{lemma}\label{lem 34}
Under the assumptions of Proposition \ref{prop 5}  and $\theta_{h-1}=E_{k}^{h-1}$, we obtain a relation of type IV.
\end{lemma}

\begin{proof}
It suffices to show that $C_{\rho}\cap C_{\sigma}\neq\emptyset$.
Suppose that $C_{\rho}\cap C_{\sigma}=\emptyset$. Let $\omega\in
C_{\sigma}$, then there exist $a_{2},\ldots,a_{h}\in E_{k}$ such
that $(\omega,a_{2},\ldots,a_{h})\not\in\rho$.
$(a_{2},\ldots,a_{h})\in E_{k}^{h-1}=\theta_{h-1}$, then there
exists $u\in E_{k}$ such that $(a_{2},\ldots,a_{h},u)\in\rho$ and
for all $2\leq i_{1}<\ldots<i_{s-1}\leq h,
(a_{i_{1}},\ldots,a_{i_{s-1}},u)\in\sigma$.

 Suppose that $(\omega,a_{2},\ldots,a_{j},u,a_{j+1},\ldots,a_{h})\in\rho$
for all $j\in\{2,\ldots,h\}$. For all $h\leq t\leq k$ we set

 $\theta_{t}=\{(b_{1},\ldots,b_{t})\in E_{k}^{t}:\exists v\in E_{k}: \forall
i_{1},\ldots,i_{h-1}\in \underline{t},
(b_{i_{1}},\ldots,b_{i_{s-1}},v)\in\sigma,$\\
$(b_{i_{1}},\ldots,b_{i_{h-1}},v)\in~\rho\}$.

Taking $v=u$, we have
$(\omega,a_{2},\ldots,a_{h})\in\theta_{h}\setminus\rho$;
 we deduce from the total symmetry of $\rho$ and $\sigma$ that $\theta_{h}$ is both totally symmetric and totally reflexive.

Suppose that $\theta_{h}\neq E_{k}^{h}$. It is easy
 to check that $\Pol\{\rho,\sigma\}\subseteq \Pol\{\rho,\theta_{h}\}\subseteq \Pol\rho$.
Let $(b_{1},\ldots,b_{h})\in E_{k}^{h}\setminus\theta_{h}$, $c\in
C_{\rho}$ and $f_{1}$ the unary operation defined by
\[
f_{1}(x)=
\begin{cases}
b_{1} & \text{if }x=\omega, \\
b_{i} & \text{if }x=a_{i}\text{ for some }2\leq i\leq h,\\
c & \text{otherwise.}
\end{cases}
\]
$(\omega,a_{2},\ldots,a_{h})\not\in\rho$ and $c\in C_{\rho}$ imply
that
 $f_{1}\in \Pol\rho$. Moreover,
$(\omega,a_{2},\ldots,a_{h})\in\theta_{h}$, but
 $(f_{1}(\omega),f_{1}(a_{2}),\ldots,f_{1}(a_{h}))=(b_{1},\ldots,b_{h})\not\in\theta_{h}$;
hence, $f_{1}\not\in \Pol\theta_{h}$. Thus,
$\Pol\{\rho,\theta_{h}\}\subsetneq \Pol\rho$.

 Let
$(a'_{1},\ldots,a'_{s})\in E_{k}^{s}\setminus\sigma$,
$(u_{1},\ldots,u_{s})\in\sigma\setminus\iota_{k}^{s}$ and
 $f_{2}$ be the unary operation defined on $E_{k}$ by
\[
f_{2}(x)=
\begin{cases}
a'_{i} & \text{if }x=u_{i}\text{ for some }1\leq i\leq s, \\
a'_{1} & \text{otherwise.}
\end{cases}
\]

We have $(u_{1},\ldots,u_{s})\in\sigma$ and
$(f_{2}(u_{1}),\ldots,f_{2}(u_{s}))=(a'_{1},\ldots,a'_{s})\not\in\sigma$;
 so, $f_{2}\not\in \Pol\sigma$. From $h>s$,
  $\img(f_{2})=\{a'_{1},\ldots,a'_{s}\}$, $\theta_{h}$
and $\rho$ are totally reflexive, we conclude that  $f_{2}\in
\Pol\{\rho,\theta_{s}\}$. Hence,
 $\Pol\{\rho,\sigma\}\subsetneq \Pol\{\rho,\theta_{s}\}\subsetneq \Pol\rho$, contradicting the maximality of $\Pol\{\rho, \sigma\}$ in $\Pol\rho$.
 Therefore, $\theta_{h}=E_{k}^{h}$. We have $\theta_{k}\neq E_{k}^{k}$ (due to $C_{\rho}\cap C_{\sigma}=\emptyset$); let $n$ be the
least integer such that $\theta_{n}\neq E_{k}^{n}$, then $n> h$. We will show that
 $\Pol\{\rho,\sigma\}\subsetneq \Pol\{\rho,\theta_{n}\}\subsetneq \Pol\rho$.

 From $\theta_{n-1}=E_{k}^{n-1}$, total reflexivity and total symmetry of $\rho$ and
 $\sigma$, we have
  total reflexivity and total symmetry of $\theta_{n}$.
 It is easy to check that $\Pol\{\rho,\sigma\}\subseteq \Pol\{\rho,\theta_{n}\}\subseteq \Pol\rho$.
 $f_{2}\in \Pol\{\rho,\theta_{n}\}$ and $f_{2}\not\in \Pol\sigma$;
hence, $\Pol\{\rho,\sigma\}\subsetneq \Pol\{\rho,\theta_{n}\}$.
Since $\theta_{n}\neq E_{k}^{n}$,
 there exists $(v_{1},\ldots,v_{n})\in E_{k}^{n}\setminus\theta_{n}$; let $(b_{1},\ldots,b_{h})\in E_{k}^{h}\setminus\rho$ and
$W^{n}_{h}=\{(i_{1},\ldots,i_{h}):1\leq i_{1}<\ldots<i_{h}\leq n\}$,
denoted for reason of notation by
 $W_{h}^{n}=\{(i_{1}^{j},\ldots,i_{h}^{j}): 1\leq j\leq q\}$.
For all $1\leq j\leq q$, we set
$\mbox{\boldmath{$y$}}_{j}=(x_{j,1},\ldots,x_{j,n})$ where
\[
 x_{j,p}=
\begin{cases}
b_{l} & \text{if }p=i_{l}^{j}\text{ for some }1\leq l\leq h, \\
b_{1} & \text{otherwise.}
\end{cases}
\]

Let $\mbox{\boldmath{$x$}}_{i}=(x_{1,i},\ldots,x_{q,i}),1\leq i\leq
n$, $c\in C_{\rho}$, and
 $f_{3}$ the  $q$-ary operation defined on $E_{k}$ by
\[
f_{3}(\mbox{\boldmath{$x$}})=
\begin{cases}
v_{i} & \text{if }\mbox{\boldmath{$x$}}=\mbox{\boldmath{$x$}}_{i}\text{ for some }1\leq i\leq n, \\
c & \text{otherwise.}
\end{cases}
\]
By the total reflexivity of $\theta_{n}$, we have
 $\{ \mbox{\boldmath{$y$}}_{1},\ldots,\mbox{\boldmath{$y$}}_{q}\}\subseteq \theta_{n}$ and
$f(\mbox{\boldmath{$y$}}_{1},\ldots,\mbox{\boldmath{$y$}}_{q})=$\\
$(f_{3}(\mbox{\boldmath{$x$}}_{1}),\ldots,f_{3}(\mbox{\boldmath{$x$}}_{n}))=(v_{1},\ldots,v_{n})\not\in\theta_{n}$.
So, $f_{3}\not\in \Pol\theta_{n}$. Moreover, for all $1\leq
i_{1}<\ldots<i_{h}\leq n,
(\mbox{\boldmath{$x$}}_{i_{1}},\ldots,\mbox{\boldmath{$x$}}_{i_{h}})\not\in\rho$.
Using this fact, we can check that $f_{3}\in \Pol\rho$.
Consequently, $\Pol\{\rho,\theta_{n}\}\varsubsetneq\Pol\rho$.
 Thus, $\Pol\{\rho,\sigma\}\varsubsetneq \Pol\{\rho,\theta_{n}\}\varsubsetneq \Pol\rho$, contradicting the maximality of
$\Pol\{\rho, \sigma\}$ in $\Pol\rho$.

Now, suppose that there exist $2\leq i_{1}<\ldots<i_{h-2}\leq h$
such that $(\omega,u,a_{i_{1}},\ldots,a_{i_{h-2}})$ is not in
$\rho$. Without loss of generality we suppose that
$(\omega,u,a_{3},\ldots,a_{h})\not\in\rho$. We set

$\gamma'_{h}=\{(b_{1},\ldots,b_{h})\in E_{k}^{h}:\exists v\in E_{k}:
(b_{2}, \ldots,b_{h},v)\in\rho \wedge\forall i_{1},
\ldots,i_{s-1}\in \underline{h}\,
(b_{i_{1}},\ldots,b_{i_{s-1}},v)\in\sigma\}$.

 Since $(\omega,u,a_{3},\ldots,a_{h})\in\gamma'_{h}\setminus\rho$ (take $v=u$), we have the following two cases: (1) $\gamma'_{h}\neq E_{k}^{h}$
and (2) $\gamma'_{h}=E_{k}^{h}$.

First, suppose that $\gamma'_{h}\neq E_{k}^{h}$. It is easy to check
that
 $\Pol\{\rho,\sigma\}\subseteq \Pol\{\rho,\gamma'_{h}\}\subseteq
 \Pol\rho$. Clearly,
 $f_{2}\not\in \Pol\sigma$ and $f_{2}\in \Pol\{\rho,\gamma'_{h}\}$ where $f_{2}$ is the function defined above.
So, $\Pol\{\rho,\sigma\}\varsubsetneq \Pol\{\rho,\gamma'_{s}\}$.
 Let $(v_{1},\ldots,v_{h})\in E_{k}^{h}\setminus\gamma'_{h}$,
 $c\in C_{\rho}$; we define the unary operation $g$ by
\[
g(x)=
\begin{cases}
v_{1} & \text{if }x=\omega, \\
v_{2} & \text{if }x=u, \\
v_{i} & \text{if }x=a_{i}\text{ for some }3\leq i\leq h, \\
c & \text{otherwise.}
\end{cases}
\]
 $(\omega,u,a_{3},\ldots,a_{h})\in\gamma'_{h}$, but
$(g(\omega),g(u),g(a_{3}),\ldots,g(a_{h}))=(v_{1},\ldots,v_{h})\not\in\gamma'_{h}$,
so $g\not\in \Pol\gamma'_{h}$. Since $c\in C_{\rho},
(\omega,u,a_{3},\ldots,a_{s})\not\in\rho$ and $\rho$ is totally
reflexive, we have $g\in \Pol\rho$. Therefore,
  $\Pol\{\rho,\gamma'_{h}\}\varsubsetneq \Pol\rho$.
 Thus, $\Pol\{\rho,\sigma\}\varsubsetneq \Pol\{\rho,\gamma'_{h}\}\varsubsetneq \Pol\rho$, contradicting the maximality of
$\Pol\{\rho, \sigma\}$ in $\Pol\rho$. So, $\gamma'_{h}=E_{k}^{h}$.

Let $F=\{\{b_{1},\ldots,b_{h-1}\}\subseteq E_{k}:
\Card(\{b_{1},\ldots,b_{h-1}\})=h-1~\mbox{and}~\{b_{1},\ldots,b_{h-1}\}\cap
C_{\rho}=\emptyset\}$. Suppose that $m=\Card(F)$,
 for all $n\in\{1,\ldots,m\}$ we set

$\gamma'_{n(h-1)+1}=\{(b_{1},\ldots,b_{n(h-1)+1})\in
E_{k}^{n(h-1)+1}: \exists v\in E_{k}: \{(b_{2},\ldots,b_{h},v),$\\
$(b_{h+1},\ldots,b_{2h-1},v),\ldots,
(b_{(n-1)(h-1)+2},\ldots,b_{n(h-1)+1},v)\}\subseteq\rho\wedge
\{(b_{i_{1}},\ldots,b_{i_{s-1}},v):
i_{1},\ldots,i_{s-1}\in\{1,\ldots,n(h-1)+1\}\subseteq\sigma\}$.

Since $C_{\rho}\cap C_{\sigma}=\emptyset$, $\gamma'_{m(h-1)+1}\neq
E_{k}^{m(h-1)+1}$. Let $n_{0}=\min\{j\geq h: \gamma'_{j(h-1)+1}\neq
E_{k}^{j(h-1)+1}\}$; then $n_{0}(h-1)+1>h$. We will show that
$\Pol\{\rho,\sigma\}\varsubsetneq
\Pol\{\rho,\gamma'_{n_{0}(h-1)+1}\}\varsubsetneq \Pol\rho$. We have
$\Pol\{\rho,\sigma\}\subseteq
\Pol\{\rho,\gamma'_{n_{0}(h-1)+1}\}\subseteq \Pol\rho$. The
operation $f_{2}$ above, preserves $\rho$ and
$\gamma'_{n_{0}(h-1)+1}$, but $f_{2}\notin \Pol\sigma$. So
$\Pol\{\rho,\sigma\}\varsubsetneq
\Pol\{\rho,\gamma'_{n_{0}(h-1)+1}\}$.

Let $(v_{1},\ldots,v_{n_{0}(h-1)+1})\in
E_{k}^{n_{0}(h-1)+1}\setminus\gamma'_{n_{0}(h-1)+1}$ and
 $W=\{(i_{1},\ldots,i_{h}): 1\leq i_{1}<\ldots<i_{h}\leq n_{0}(h-1)+1\}$ denoted for the reason of notation by
$W=\{(i_{1}^{j},\ldots,i_{h}^{j}):1\leq j\leq q\}$. For all $1\leq
j\leq q$, set
$\mbox{\boldmath{$y$}}_{j}=(x_{j,1},\ldots,x_{j,n_{0}(h-1)+1})$ with
\[
x_{j,p}=
\begin{cases}
\omega & \text{if }p=i_{1}^{j}, \\
a_{l} & \text{if }p=i_{l}^{j}\text{ for some }3\leq l\leq h, \\
u & \text{otherwise.}
\end{cases}
\]
 where $u$ is a fixed element such that
$(a_{2},\ldots,a_{h},u)\in\rho$,
$(a_{i_{1}},\ldots,a_{i_{s-1}},u)\in\sigma$ for
$i_{1},\ldots,i_{s-1}\in\{2,\ldots,h\}$ and $\omega\in C_{\sigma}$.
From $(u,a_{3},\ldots,a_{h},\omega)\not\in\rho$, we have
$\Card(\{\omega,u,a_{3},\ldots,a_{h}\})=~h$. Let
$\mbox{\boldmath{$x$}}_{i}=(x_{1,i},\ldots,x_{q,i}),1 \leq i\leq
n_{0}(h-1)+1$. We can see that
 $(\mbox{\boldmath{$x$}}_{i_{1}},\ldots,\mbox{\boldmath{$x$}}_{i_{h}})\not\in\rho$ for all
$1\leq i_{1}<\ldots<i_{h}\leq n_{0}(h-1)+1$. Let
$(v_{1},\ldots,v_{n_{0}(h-1)+1})\in
E_{k}^{n_{0}(h-1)+1}\setminus~\gamma'_{n_{0}(h-1)+1}$; consider the
$q$-ary operation $f$ defined on $E_{k}$ by
\[
f(\mbox{\boldmath{$x$}})=
\begin{cases}
v_{i} & \text{if }\mbox{\boldmath{$x$}}=\mbox{\boldmath{$x$}}_{i}\text{ for some }1 \leq i\leq n_{0}(h-1)+1,\\
c & \text{otherwise.}
\end{cases}
\]
By construction,
$\{\mbox{\boldmath{$y$}}_{1},\ldots,\mbox{\boldmath{$y$}}_{q}\}\subseteq\gamma'_{n_{0}(h-1)+1}$;
furthermore,
 $f(\mbox{\boldmath{$y$}}_{1},\ldots,\mbox{\boldmath{$y$}}_{q})=(f(\mbox{\boldmath{$x$}}_{1}),$\\
 $\ldots,
 f(\mbox{\boldmath{$x$}}_{n_{0}(h-1)+1}))=(v_{1},\ldots,v_{n_{0}(h-1)+1})\not\in\gamma'_{n_{0}(h-1)+1}$;
  so, $f\not\in \Pol\gamma'_{n_{0}(h-1)+1}$.
Using total reflexivity and total symmetry of $\rho$, and the fact
that
$(\mbox{\boldmath{$x$}}_{i_{1}},\ldots,\mbox{\boldmath{$x$}}_{i_{h}})\not\in\rho$
for $1\leq i_{1}<,\cdots,<i_{h}\leq n_{0}(h-1)+1$,
 we can show that $f\in \Pol\rho$.
 So, $\Pol\{\rho,\gamma'_{n_{0}(h-1)+1}\}\varsubsetneq \Pol\rho$.
Thus,
$\Pol\{\rho,\sigma\}\varsubsetneq\Pol\{\rho,\gamma'_{n_{0}(h-1)+1}\}\varsubsetneq
\Pol\rho$, contradicting the maximality of $\Pol\{\rho, \sigma\}$ in
$\Pol\rho$.
\end{proof}

\begin{proof}[Proof (of Proposition \ref{prop 5})]
Combining the Lemmas \ref{lem 31}--\ref{lem 34} we have the proof of
Proposition~\ref{prop 5}.
\end{proof}

\subsection*{Acknowledgements}: The authors would like to express
their thanks to the referees for their comments and suggestions
which improved the paper.

\end{document}